\documentclass[11pt,a4paper]{article}

\usepackage{latexsym}
\usepackage{amsmath}
\usepackage{amssymb}
\usepackage{amscd}
\usepackage{amsthm}
\usepackage[all]{xy}

\newtheorem{thm}{Theorem}[section]
\newtheorem{cor}[thm]{Corollary}
\newtheorem{lem}[thm]{Lemma}
\newtheorem{prop}[thm]{Proposition}

\theoremstyle{definition}
\newtheorem{defn}[thm]{Definition}
\newtheorem{rem}[thm]{Remark}

\newcommand{\R}{\mathbb R}
\newcommand{\Z}{\mathbb Z}

\newcommand{\C}{\mathbb C}

\newif\ifpdf \pdftrue
\ifx\pdfoutput\undefined \pdffalse \fi \ifx\pdfoutput\relax
\pdffalse \fi

\usepackage{makeidx}

\makeindex

\begin{document}

\title{Concentration theorem and relative fixed point formula of Lefschetz type in Arakelov
geometry}

\author{Shun Tang}

\date{}

\maketitle

\vspace{-10mm}

\hspace{5cm}\hrulefill\hspace{5.5cm} \vspace{5mm}

\textbf{Abstract.} In this paper we prove a concentration theorem
for arithmetic $K_0$-theory, this theorem can be viewed as an
analog of R. Thomason's result (cf. \cite{Th}) in the arithmetic
case. We will use this arithmetic concentration theorem to prove a
relative fixed point formula of Lefschetz type in the context of
Arakelov geometry. Such a formula was conjectured of a slightly
stronger form by K. K\"{o}hler and D. Roessler in \cite{KR2} and
they also gave a correct route of its proof there. Nevertheless
our new proof is much simpler since it looks more natural and it
doesn't involve too many complicated computations.


\tableofcontents

\section{Introduction}
In \cite{KR1}, K. K\"{o}hler and D. Roessler proved a Lefschetz
type fixed point formula for regular schemes endowed with the
action of a diagonalisable group scheme, this formula generalized
the regular case of P. Baum, W. Fulton and G. Quart's result in
\cite{BFQ} to Arakelov geometry. To make things more explicit, we
first recall the main theorem in \cite{KR1}, all notations will be
explained in the text.

Let $D$ be an arithmetic ring, and let $\mu_n$ be the
diagonalisable group scheme over $D$ associated to the cyclic
group $\Z/{n\Z}$. Denote the ring $K_0(\Z)[\Z/{n\Z}]\cong
\Z[T]/{(1-T^n)}$ by $R(\mu_n)$. Let $\rho$ be the prime ideal in
$R(\mu_n)$ which is the kernel of the canonical morphism
$\Z[T]/{(1-T^n)}\rightarrow \Z[T]/{(\Phi_n)}$ where $\Phi_n$
stands for the $n$-th cyclotomic polynomial. By construction the
elements $1-T^k$ for $k=1,\ldots,n-1$ are not contained in $\rho$.
To every $\mu_n$-equivariant arithmetic variety $X$, we can
associate an equivariant arithmetic $K_0$-group
$\widehat{K_0}(X,\mu_n)$ which contains certain set of smooth form
classes on $X_{\mu_n}(\C)$ as analytic datum where $X_{\mu_n}$
stands for the fixed point subscheme under the action of $\mu_n$.
Such an equivariant arithmetic $K_0$-group has a ring structure
and moreover it can be made to be an $R(\mu_n)$-algebra. Let $f$
be the structure morphism of $X$ over $D$, and let
$\overline{N}_{X/{X_{\mu_n}}}$ be the normal bundle with respect
to the regular immersion $X_{\mu_n}\hookrightarrow X$, endowed
with the quotient metric induced by a chosen K\"{a}hler metric of
$X(\C)$. Then the main theorem in \cite{KR1} reads: the element
$\lambda_{-1}(\overline{N}_{X/{X_{\mu_n}}}^\vee):=\Sigma_{j=0}^{{\rm
rk}(N_{X/{X_{\mu_n}}}^\vee)}(-1)^j\Lambda^j(\overline{N}_{X/{X_{\mu_n}}}^\vee)$
is invertible in $\widehat{K_0}(X_{\mu_n},\mu_n)_{\rho}$, the
localization of $\widehat{K_0}(X_{\mu_n},\mu_n)$ with respect to
the ideal $\rho$, and we have the following commutative diagram
\begin{displaymath}
\xymatrix{
\widehat{K_0}(X,\mu_n) \ar[rr]^-{\Lambda_R(f)^{-1}\cdot \tau} \ar[d]_{f_*} && \widehat{K_0}(X_{\mu_n},\mu_n)_{\rho} \ar[d]^{{f_{\mu_n}}_*} \\
 \widehat{K_0}(D,\mu_n) \ar[rr]^-{\iota} && \widehat{K_0}(D,\mu_n)_{\rho}}
\end{displaymath} where
$\Lambda_R(f):=\lambda_{-1}(\overline{N}_{X/{X_{\mu_n}}}^\vee)\cdot(1+R_g(N_{X/{X_{\mu_n}}}))$,
$\tau$ stands for the restriction map and $\iota$ is the natural
morphism from a ring or a module to its localization which sends
an element $e$ to $\frac{e}{1}$. Here $R_g(\cdot)$ is the
equivariant $R$-genus, the definition of the two push-forward
morphisms $f_*$ and ${f_{\mu_n}}_*$ involves an important analytic
datum which is called the equivariant analytic torsion.

In \cite{Ma1}, X. Ma defined the equivariant analytic torsion form
which is a higher analog of the equivariant analytic torsion and
he also proved the curvature and anomaly formulae for it. Once
these preparations are done, one can naturally conjecture a
Lefschetz fixed point formula which generalizes K. K\"{o}hler and
D. Roessler's results to the relative setting. Precisely speaking,
let $X$ and $Y$ be two $\mu_n$-equivariant arithmetic varieties,
suppose that we are given a flat $\mu_n$-equivariant morphism $f:
X\rightarrow Y$ which is smooth on the complex numbers, then we
can define reasonable push-forward morphisms $f_*$ and
${f_{\mu_n}}_*$ by using equivariant analytic torsion forms.
Correspondingly, by defining a suitable element $M(f)$ in the
localization $\widehat{K_0}(X_{\mu_n},\mu_n)_{\rho}$, we will get
the following commutative diagram
\begin{displaymath}
\xymatrix{
\widehat{K_0}(X,\mu_n) \ar[rr]^-{M(f)\cdot \tau} \ar[d]_{f_*} && \widehat{K_0}(X_{\mu_n},\mu_n)_{\rho} \ar[d]^{{f_{\mu_n}}_*} \\
 \widehat{K_0}(Y,\mu_n) \ar[rr]^-{\tau} &&
\widehat{K_0}(Y_{\mu_n},\mu_n)_{\rho}}
\end{displaymath}
where $\tau$ is again the restriction map.

There are two crucial actors in K. K\"{o}hler and D. Roessler's
proof of their main theorem. One is an elegant algebro-geometric
construction: the deformation to the normal cone, and the other
one is a technical result: $\widehat{K_0}(\cdot,\mu_n)$-theoretic
form of Bismut's immersion formula. The construction of the
deformation to the normal cone also works in the relative setting,
so if one can find a general Bismut's immersion formula in the
case of relative setting, the second commutative diagram described
above is no longer a conjecture but a theorem. J.-M. Bismut and X.
Ma fulfilled this work in \cite{BM}.

K. K\"{o}hler and D. Roessler's proof is a little hard to follow
since it involves too many complicated computations, we will
provide a much simpler proof in this paper but on one more
condition: the fibre product $f^{-1}(Y_{\mu_n})$ is also regular.
This limitation will not influence most applications in practice
(cf. Remark~\ref{505} (i)). Especially, the main results in
\cite{KR1} can be totally recovered. Our proof has nothing to do
with the construction of the deformation to the normal cone, it
relies on an arithmetic concentration theorem in Arakelov
geometry. This concentration theorem is an analog of R. Thomason's
result (cf. \cite{Th}) at $0$-degree level. Consider the closed
immersion $i: X_{\mu_n}\hookrightarrow X$, R. Thomason used the
Quillen's localization sequence for higher equivariant $K$-theory
to prove that the natural morphisms
\begin{displaymath}
i_*:\quad K_*(X_{\mu_n},\mu_n)_{\rho}\rightarrow
K_*(X,\mu_n)_{\rho}
\end{displaymath} are all isomorphisms and
their inverse morphisms are of the form
$\lambda_{-1}^{-1}(N_{X/{X_{\mu_n}}}^\vee)\cdot i^*$. In the
arithmetic case, we will first construct the morphism
\begin{displaymath}
i_*:\quad \widehat{K_0}(X_{\mu_n},\mu_n)_{\rho}\rightarrow
\widehat{K_0}(X,\mu_n)_{\rho}
\end{displaymath} and then prove
that it is also an isomorphism. The most important result we need
to prove the relative Lefschetz fixed point formula in Arakelov
geometry is that the inverse morphism of $i_*$ is given by
$\lambda_{-1}^{-1}(\overline{N}_{X/{X_{\mu_n}}}^\vee)\cdot i^*$.

The structure of this paper is as follows. In Section 2 and
Section 3, we recall the algebraic concentration theorem and the
equivariant arithmetic $K_0$-theory for the convenience of the
reader. In Section 4, we will discuss all analytic results coming
from differential geometry which are needed in this paper. In
Section 5 and Section 6, we describe and prove the arithmetic
concentration theorem and the relative Lefschetz fixed point
formula.

\textbf{Acknowledgements.} The author wishes to thank his thesis
advisor Damian Roessler for his constant encouragement and for
many fruitful discussions between them. Especially, Damian
Roessler is the person who made the author realize that it is
possible to prove the relative Lefschetz fixed point formula via
proving an arithmetic concentration theorem. The author also
wishes to thank Xiaonan Ma for his kindly explanation of various
purely analytic problems as well as his useful comments concerning
a crucial lemma in this paper. Finally, thanks to the referee,
whose very detailed comments help to improve the paper a great
deal.

\section{Algebraic concentration theorem}
Let $D$ be a Noetherian integral ring. In this section we fix
$S:={\rm Spec}(D)$ as the base scheme. Let $n$ be a positive
integer, we shall denote by $\mu_n$ the diagonalisable group
scheme over $S$ associated to the cyclic group $\Z/{n\Z}$. By a
$\mu_n$-equivariant scheme we understand a separable and of finite
type scheme over $S$ which admits a $\mu_n$-action. A
$\mu_n$-action on a scheme $X$ is a morphism $m_X: \mu_n\times
X\to X$ which statisfies some compatibility properties. Denote by
$p_X$ the projection from $\mu_n\times X$ to $X$. For a coherent
$\mathcal{O}_X$-module $E$ on $X$, a $\mu_n$-action on $E$ we mean
an isomorphism of coherent sheaves $m_E: p_X^*E\to m_X^*E$ which
satisfies certain associativity properties. We refer to \cite{Ko}
and \cite{KR1}, Section 2 for the group scheme action theory we
are talking about here.

Let $X$ be a $\mu_n$-equivariant scheme, we consider the category
of coherent $\mathcal{O}_X$-modules endowed with an action of
$\mu_n$ which are compatible with the $\mu_n$-structure of $X$.
According to Quillen, to this category we may associate a graded
abelian group $G_*(X,\mu_n)$ which is called the higher algebraic
equivariant $G$-group. If one replaces the $\mu_n$-equivariant
coherent $\mathcal{O}_X$-modules by the $\mu_n$-equivariant vector
bundles of finite rank, one gets the higher algebraic equivariant
$K$-group $K_*(X,\mu_n)$. It is well known that the tensor product
of $\mu_n$-equivariant vector bundles induces a graded ring
structure on $K_*(X,\mu_n)$ and a graded $K_*(X,\mu_n)$-module
structure on $G_*(X,\mu_n)$. Notice that if $X$ is regular, then
the natural morphism from $K_*(X,\mu_n)$ to $G_*(X,\mu_n)$ is an
isomorphism.

Denote by $X_{\mu_n}$ the fixed point subscheme of $X$ under the
action of $\mu_n$, then the closed immersion $i:
X_{\mu_n}\hookrightarrow X$ induces two group homomorphisms $i_*:
G_*(X_{\mu_n},\mu_n)\to G_*(X,\mu_n)$ and $i_*:
K_*(X_{\mu_n},\mu_n)\to K_*(X,\mu_n)$ which satisfy the projection
formula. According to \cite{SGA3}, I 4.4, $\mu_n$ is the pull-back
of a unique diagonalisable group scheme over $\Z$ associated to
the same group, this group scheme will be still denoted by
$\mu_n$. Write $R(\mu_n)$ for the group $K_0(\Z,\mu_n)$ which is
isomorphic to $\Z[T]/{(1-T^n)}$. Let $\rho$ be the prime ideal of
$R(\mu_n)$ which is defined to be the kernel of the following
canonical morphism
\begin{displaymath}
\Z[T]/{(1-T^n)}\to \Z[T]/{(\Phi_n)}
\end{displaymath}
where $\Phi_n$ stands for the $n$-th cyclotomic polynomial. The
prime ideal $\rho$ is chosen to satisfy the condition that the
localization $R(\mu_n)_{\rho}$ is a $R(\mu_n)$-algebra in which
the elements $1-T^k$ from $k=1$ to $n-1$ are all invertible. This
condition plays a crucial role in the proof of the concentration
theorem. If the $\mu_n$-equivariant scheme $X$ is regular, then
$X_{\mu_n}$ is also regular. We shall write
$\lambda_{-1}(N_{X/{X_{\mu_n}}}^\vee)$ for the alternating sum
$\sum(-1)^j\wedge^jN_{X/{X_{\mu_n}}}^\vee$ where
$N_{X/{X_{\mu_n}}}$ stands for the normal bundle associated to the
regular immersion $i$. Then the algebraic concentration theorem in
\cite{Th} can be described as the following.
\begin{thm}(Thomason)\label{concentration}
Let notations and assumptions be as above.
\begin{itemize}
\item The $R(\mu_n)_{\rho}$-module morphism $i_*:
G_*(X_{\mu_n},\mu_n)_{\rho}\to G_*(X,\mu_n)_{\rho}$ is actually an
isomorphism. \item If $X$ is regular, then
$\lambda_{-1}(N_{X/{X_{\mu_n}}}^\vee)$ is invertible in
$G_*(X_{\mu_n},\mu_n)_{\rho}$ and the inverse map of $i_*$ is
given by $\lambda_{-1}^{-1}(N_{X/{X_{\mu_n}}}^\vee)\cdot i^*$.
\end{itemize}
\end{thm}

The proof of Thomason's algebraic concentration theorem can be
split into three steps. The first step is to show that
$G_*(U,\mu_n)_{\rho}\cong 0$ if $U$ has no fixed point, then the
claim that $i_*: G_*(X_{\mu_n},\mu_n)_{\rho}\to
G_*(X,\mu_n)_{\rho}$ is an isomorphism follows from Quillen's
localization sequence for higher equivariant $K$-theory, see
\cite{Th}, Th\'{e}or\`{e}me 2.1. The second step is to show that
$\lambda_{-1}(N_{X/{X_{\mu_n}}}^\vee)$ is invertible in
$G_*(X_{\mu_n},\mu_n)_{\rho}$ if $X$ is regular (cf. \cite{Th},
Lemme 3.2). The last step is a direct computation using the
projection formula for equivariant $K$-theory to show that the
inverse map of $i_*$ is exactly
$\lambda_{-1}^{-1}(N_{X/{X_{\mu_n}}}^\vee)\cdot i^*$ (cf.
\cite{Th}, Lemme 3.3). The condition that the localization
$R(\mu_n)_{\rho}$ is a $R(\mu_n)$-algebra in which the elements
$1-T^k$ from $k=1$ to $n-1$ are all invertible was used in the
first and the second step.

\section{Equivariant arithmetic $K_0$-theory}
By an arithmetic ring $D$ we understand a regular, excellent,
Noetherian integral ring, together with a finite set $\mathcal{S}$
of embeddings $D\hookrightarrow \C$, which is invariant under a
conjugate-linear involution $F_\infty$ (cf. \cite{GS}, Def.
3.1.1). A $\mu_n$-equivariant arithmetic variety $X$ over $D$ is a
regular Noetherian scheme $X$ endowed with a $\mu_n$-projective
action over $D$, namely there exists an equivariant closed
immersion from $X$ to some $\mu_n$-equivariant projective space
$\mathbb{P}_D^n$ (cf. \cite{KR1}, Definition 2.3). Let $X$ be a
$\mu_n$-equivariant arithmetic variety, then $X(\C)$, the set of
complex points of the variety $\coprod_{\sigma\in
\mathcal{S}}X\times_D\C$, is a compact complex manifold. This
manifold admits an action of the group of complex $n$-th roots of
unity and an anti-holomorphic involution induced by $F_\infty$
which is still denoted by $F_\infty$. It was shown in \cite{Th},
Prop. 3.1 that the fixed point subscheme $X_{\mu_n}$ is also
regular. Fix a primitive $n$-th root of unity $\zeta_n$ and denote
its corresponding holomorphic automorphism on $X(\C)$ by $g$, by
GAGA principle we have a natural isomorphism $X_{\mu_n}(\C)\cong
X(\C)_g$. Moreover, denote by $A^{p,p}(X(\C)_g)$ the set of smooth
forms $\omega$ of type $(p,p)$ on $X(\C)_g$ which satisfy
$F_\infty^*\omega=(-1)^p\omega$, we shall write
$\widetilde{A}(X_{\mu_n})$ for the set of form classes
\begin{displaymath}
\widetilde{A}(X(\C)_g):=\bigoplus_{p\geq 0}(A^{p,p}(X(\C)_g)/({\rm
Im}\partial+{\rm Im}\overline{\partial})).
\end{displaymath}
Similarly, denote by $D^{p,p}(X(\C)_g)$ the set of currents $T$ of
type $(p,p)$ on $X(\C)_g$ which satisfy $F_\infty^*T=(-1)^pT$, we
shall write $\widetilde{\mathcal{U}}(X_{\mu_n})$ for the set of
current classes
\begin{displaymath}
\widetilde{\mathcal{U}}(X(\C)_g):=\bigoplus_{p\geq
0}(D^{p,p}(X(\C)_g)/({\rm Im}\partial+{\rm
Im}\overline{\partial})).
\end{displaymath}

\begin{defn}\label{201}
An equivariant hermitian vector bundle $\overline{E}$ on $X$ is a
hermitian vector bundle $\overline{E}$ on $X$, endowed with a
$\mu_n$-action which lifts the action of $\mu_n$ on $X$ such that
the hermitian metric on $E_\C$ is invariant under $F_\infty$ and
$g$.
\end{defn}

\begin{rem}\label{202}
Let $\overline{E}$ be an equivariant hermitian vector bundle on
$X$, the restriction of $\overline{E}$ to the fixed point
subscheme $X_{\mu_n}$ has a natural $\Z/{n\Z}$-grading structure
$\overline{E}\mid_{X_{\mu_n}}\cong \oplus_{k\in
\Z/{n\Z}}\overline{E}_k$. $E_k$ is the subbundle of
$E\mid_{X_{\mu_n}}$ such that for open affine subscheme $V$ of
$X_{\mu_n}$ the action of $\mu_n(V)$ on $E_k(V)$ is given by
$u\bullet e_k=u^k\cdot e_k$ (cf. \cite{SGA3}, I, Prop. 4.7.3).  We
shall often write $\overline{E}_{\mu_n}$ for $\overline{E}_0$.
\end{rem}

Following the same notations and definitions as in \cite{KR1},
Section 3, we write ${\rm ch}_g(\overline{E})$ for the equivariant
Chern character form
\begin{displaymath}
{\rm ch}_g((E_\C,h))=\sum_{k\in \Z/{n\Z}}\zeta_n^k{\rm
ch}({E_k}_\C,h_{E_k})
\end{displaymath} associated to the
hermitian holomorphic vector bundle $(E_\C,h)$ on $X(\C)$.
Moreover, we have the equivariant Todd form
\begin{displaymath}
{\rm Td}_g(\overline{E}):=\frac{c_{{\rm
rk}{E_{\mu_n}}_\C}({\overline{E}_{\mu_n}}_\C)}{{\rm
ch}_g(\sum_{j=0}^{{\rm rk}E}(-1)^j\wedge^j\overline{E}^\vee)}.
\end{displaymath}
Furthermore, let $\overline{\varepsilon}: 0\rightarrow
\overline{E}'\rightarrow \overline{E}\rightarrow
\overline{E}''\rightarrow 0$ be an exact sequence of equivariant
hermitian vector bundles on $X$, we can associate to it an
equivariant Bott-Chern secondary characteristic class
$\widetilde{{\rm ch}}_g(\overline{\varepsilon})\in
\widetilde{A}(X_{\mu_n})$ which satisfies the differential
equation
\begin{displaymath}
{\rm dd}^c\widetilde{{\rm ch}}_g(\overline{\varepsilon})={\rm
ch}_g(\overline{E}')-{\rm ch}_g(\overline{E})+{\rm
ch}_g(\overline{E}'')
\end{displaymath} where ${\rm dd}^c$ is the
differential operator $\frac{\overline{\partial}\partial}{2\pi
i}$. Similarly, we have the equivariant Todd secondary
characteristic class $\widetilde{{\rm
Td}}_g(\overline{\varepsilon})\in \widetilde{A}(X_{\mu_n})$ which
satisfies the differential equation
\begin{displaymath}
{\rm dd}^c\widetilde{{\rm Td}}_g(\overline{\varepsilon})={\rm
Td}_g(\overline{E}'){\rm Td}_g(\overline{E}'')-{\rm
Td}_g(\overline{E}).
\end{displaymath}
Let $E$ be an equivariant vector bundle with two different
hermitian metrics $h_1$ and $h_2$, we shall write $\widetilde{{\rm
ch}}_g(E,h_1,h_2)$ (resp. $\widetilde{{\rm Td}}_g(E,h_1,h_2)$) for
the equivariant Bott-Chern (resp. Todd) secondary characteristic
class associated to the exact sequence
\begin{displaymath}
0\to (E,h_1)\to (E,h_2)\to 0\to 0
\end{displaymath} where the map
from $(E,h_1)$ to $(E,h_2)$ is the identity map.

\begin{defn}\label{203}
Let $X$ be a $\mu_n$-equivariant arithmetic variety, we define the
equivariant arithmetic Grothendieck group $\widehat{K_0}(X,\mu_n)$
with respect to $X$ as the free abelian group generated by the
elements of $\widetilde{A}(X_{\mu_n})$ and by the equivariant
isometry classes of equivariant hermitian vector bundles on $X$,
together with the relations

(i). for every exact sequence $\overline{\varepsilon}$ as above,
$\widetilde{{\rm
ch}}_g(\overline{\varepsilon})=\overline{E}'-\overline{E}+\overline{E}''$;

(ii). if $\alpha\in \widetilde{A}(X_{\mu_n})$ is the sum of two
elements $\alpha'$ and $\alpha''$ in $\widetilde{A}(X_{\mu_n})$,
then the equality $\alpha=\alpha'+\alpha''$ holds in
$\widehat{K_0}(X,\mu_n)$.
\end{defn}

\begin{rem}\label{204}
The definition of the arithmetic $K_0$-group implies that there is
an exact sequence
\begin{displaymath}
\xymatrix{
 \widetilde{A}(X_{\mu_n})
\ar[r]^-{a} & \widehat{K_0}(X,\mu_n) \ar[r]^-{\pi} & K_0(X,\mu_n)
\ar[r] & 0}
\end{displaymath} where $a$ is the
natural map which sends $\alpha\in \widetilde{A}(X_{\mu_n})$ to
the class of $\alpha$ in $\widehat{K_0}(X,\mu_n)$ and $\pi$ is the
forgetful map.
\end{rem}

We now describe the ring structure of $\widehat{K_0}(X,\mu_n)$. We
consider the generators of the abelian group
$\widehat{K_0}(X,\mu_n)$, for two equivariant hermitian vector
bundles $\overline{E}$, $\overline{E}'$ on $X$ and two elements
$\alpha$, $\alpha'$ in $\widetilde{A}(X_{\mu_n})$, we define the
rules of the product $\cdot$ as $\overline{E}\cdot
\overline{E}':=\overline{E}\otimes \overline{E}'$,
$\overline{E}\cdot \alpha=\alpha\cdot \overline{E}:={\rm
ch}_g(\overline{E})\wedge\alpha$ and $\alpha\cdot\alpha':={\rm
dd}^c\alpha\wedge\alpha'$. Note that $\alpha$ and $\alpha'$ are
both smooth, so $\alpha\cdot\alpha'$ is well-defined and it is
commutative in $\widetilde{A}(X_{\mu_n})$. It is easy to verify
that our definition is compatible with the two generating
relations in Definition~\ref{203}, we leave the verification to
the reader.

Furthermore, recall that $R(\mu_n)=\Z[T]/(1-T^n)$. Let
$\overline{I}$ be the $\mu_n$-equivariant hermitian projective
$D$-module whose term of degree $1$ is $D$ endowed with the
trivial metric and whose other terms are $0$, namely $I$ is the
structure sheaf $\mathcal{O}_D$ of ${\rm Spec}(D)$ on which the
$\mu_n$-action is given by $u\bullet d=u\cdot d$ where $u\in
\mu_n(V), d\in \mathcal{O}_D(V)$ and $V$ is an open affine
subscheme of ${\rm Spec}(D)$. Then we may make
$\widehat{K_0}(D,\mu_n)$ an $R(\mu_n)$-algebra under the ring
morphism which sends $T$ to $\overline{I}$. By doing pull-backs,
we may endow every arithmetic Grothendieck group we defined before
with an $R(\mu_n)$-module structure.

Since the classical arguments of locally free resolutions may not
be compatible with the equivariant setting, we summarize some
crucial facts we need as follows.

(i). Every equivariant coherent sheaf on an equivariant arithmetic
variety is an equivariant quotient of an equivariant locally free
coherent sheaf.

(ii). Every equivariant coherent sheaf on an equivariant
arithmetic variety admits a finite equivariant locally free
resolution.

(iii). An exact sequence of equivariant coherent sheaves on an
equivariant arithmetic variety admits an exact sequence of
equivariant locally free resolutions.

(iv). Any two equivariant locally free resolutions of an
equivariant coherent sheaf on an equivariant arithmetic variety
can be dominated by a third one.

All these statements can be found in the proof of \cite{KR1},
Prop. 4.2. For (i) and (ii), one can also see \cite{Ko}, Remark
3.5.

To end this section, we recall the following important lemma which
will be used frequently in this paper.
\begin{lem}\label{205}(\cite{KR1}, Lemma 4.5)
Let $X$ be a $\mu_n$-equivariant arithmetic variety and let
$\overline{E}$ be an equivariant hermitian vector bundle on
$X_{\mu_n}$ such that $\overline{E}_{\mu_n}=0$. Then the element
$\lambda_{-1}(\overline{E})$ is invertible in
$\widehat{K_0}(X_{\mu_n},\mu_n)_{\rho}$.
\end{lem}

\section{Analytical preliminaries}

\subsection{Equivariant analytic torsion forms}
In \cite{BK}, J.-M. Bismut and K. K\"{o}hler extended the
Ray-Singer analytic torsion to the higher analytic torsion form
$T$ for a holomorphic submersion. The differential equation on
${\rm dd}^cT$ gives a refinement of the Grothendieck-Riemann-Roch
theorem. And later, X. Ma generalized J.-M. Bismut and K.
K\"{o}hler's results to the equivariant case in his article
\cite{Ma1}. In this subsection, we shall recall Ma's construction
of the equivariant analytic torsion form since it will be used to
define a reasonable push-forward morphism between equivariant
arithmetic $K_0$-groups.

Let $f: M\rightarrow B$ be a proper holomorphic submersion of
complex manifolds, and let $TM$, $TB$ be the holomorphic tangent
bundle on $M$, $B$. Denote by $J^{Tf}$ the complex structure on
the real relative tangent bundle $T_{\R}f$, and assume that
$h^{Tf}$ is a hermitian metric on $Tf$ which induces a Riemannian
metric $g^{Tf}$. Let $T^HM$ be a vector subbundle of $TM$ such
that $TM=T^HM\oplus Tf$, we first give the definition of
K\"{a}hler fibration as in \cite{BGS2}, Def. 1.4.

\begin{defn}\label{311}
The triple $(f,h^{Tf},T^HM)$ is said to define a K\"{a}hler
fibration if there exists a smooth real $(1,1)-$form $\omega$
which satisfies the following three conditions:

(i). $\omega$ is closed;

(ii). $T_{\R}^HM$ and $T_{\R}f$ are orthogonal with respect to
$\omega$;

(iii). if $X,Y\in T_{\R}f$, then $\omega(X,Y)=\langle
X,J^{Tf}Y\rangle_{g^{Tf}}$.
\end{defn}

It was shown in \cite{BGS2}, Thm. 1.5 and 1.7 that for a given
K\"{a}hler fibration, the form $\omega$ is unique up to addition
of a form $f^*\eta$ where $\eta$ is a real, closed $(1,1)$-form on
$B$. Moreover, for any real, closed $(1,1)$-form $\omega$ on $M$
such that the bilinear map $X,Y\in T_{\R}f\mapsto
\omega(J^{Tf}X,Y)\in \R$ defines a Riemannian metric and hence a
hermitian product $h^{Tf}$ on $Tf$, we can define a K\"{a}hler
fibration whose associated $(1,1)$-form is $\omega$. In
particular, for a given $f$, a K\"{a}hler metric on $M$ defines a
K\"{a}hler fibration if we choose $T^HM$ to be the orthogonal
complement of $Tf$ in $TM$ and $\omega$ to be the K\"{a}hler form
associated to this metric.

We now recall the Bismut superconnection of a K\"{a}hler
fibration. Let $(\xi,h^\xi)$ be a hermitian holomorphic vector
bundle on $M$. Let $\nabla^{Tf}$, $\nabla^\xi$ be the holomorphic
hermitian connections on $(Tf,h^{Tf})$ and $(\xi,h^\xi)$. Let
$\nabla^{\Lambda(T^{*(0,1)}f)}$ be the connection induced by
$\nabla^{Tf}$ on $\Lambda(T^{*(0,1)}f)$. Then we may define a
connection on $\Lambda(T^{*(0,1)}f)\otimes\xi$ by setting
\begin{displaymath}
\nabla^{\Lambda(T^{*(0,1)}f)\otimes\xi}=\nabla^{\Lambda(T^{*(0,1)}f)}\otimes
1+1\otimes\nabla^\xi.
\end{displaymath} Let $E$ be the
infinite-dimensional bundle on $B$ whose fibre at each point $b\in
B$ consists of the $C^\infty$ sections of
$\Lambda(T^{*(0,1)}f)\otimes\xi\mid_{f^{-1}b}$. This bundle $E$ is
a smooth $\Z$-graded bundle. We define a connection $\nabla^E$ on
$E$ as follows. If $U\in T_{\R}B$, let $U^H$ be the lift of $U$ in
$T^H_{\R}M$ so that $f_*U^H=U$. Then for every smooth section $s$
of $E$ over $B$, we set
\begin{displaymath}
\nabla_U^Es=\nabla_{U^H}^{\Lambda(T^{*(0,1)}f)\otimes\xi}s.
\end{displaymath} For $b\in B$, let $\overline{\partial}^{Z_b}$ be
the Dolbeault operator acting on $E_b$, and let
$\overline{\partial}^{Z_b*}$ be its formal adjoint with respect to
the canonical hermitian product on $E_b$ (cf. \cite{Ma1}, 1.2).
Let $C(T_{\R}f)$ be the Clifford algebra of $(T_{\R}f,h^{Tf})$,
then the bundle $\Lambda(T^{*(0,1)}f)\otimes\xi$ has a natural
$C(T_{\R}f)$-Clifford module structure. Actually, if $U\in Tf$,
let $U'\in T^{*(0,1)}f$ correspond to $U$ defined by
$U'(\cdot)=h^{Tf}(U,\cdot)$, then for $U, V\in Tf$ we set
\begin{displaymath}
c(U)=\sqrt{2}U'\wedge,\quad
c(\overline{V})=-\sqrt{2}i_{\overline{V}}
\end{displaymath} where
$i_{(\cdot)}$ is the contraction operator (cf. \cite{BGV},
Definition 1.6). Moreover, if $U,V\in T_{\R}B$, we set
$T(U^H,V^H)=-P^{Tf}[U^H,V^H]$ where $P^{Tf}$ stands for the
canonical projection from $TM$ to $Tf$.

\begin{defn}\label{312}
Let $e_1,\ldots,e_{2m}$ be a basis of $T_{\R}B$, and let
$e^1,\ldots,e^{2m}$ be the dual basis of $T_{\R}^*B$. Then the
element
\begin{displaymath}
c(T)=\frac{1}{2}\sum_{1\leq\alpha,\beta\leq 2m}e^\alpha\wedge
e^\beta\widehat{\otimes}c(T(e_\alpha^H,e_\beta^H))
\end{displaymath} is a section of
$(f^*\Lambda(T^*_{\R}B)\widehat{\otimes}{\rm
End}(\Lambda(T^{*(0,1)}f)\otimes\xi))^{\rm odd}$.
\end{defn}

\begin{defn}\label{313}
For $u>0$, the Bismut superconnection on $E$ is the differential
operator
\begin{displaymath}
B_u=\nabla^E+\sqrt{u}(\overline{\partial}^Z+\overline{\partial}^{Z*})-\frac{1}{2\sqrt{2u}}c(T)
\end{displaymath} on
$f^*(\Lambda(T_{\R}^*B))\widehat{\otimes}(\Lambda(T^{*(0,1)}f)\otimes\xi)$.
\end{defn}

\begin{defn}\label{314}
Let $N_V$ be the number operator on
$\Lambda(T^{*(0,1)}f)\otimes\xi$ and on $E$, namely $N_V$ acts as
multiplication by $p$ on $\Lambda^p(T^{*(0,1)}f)\otimes\xi$. For
$U,V\in T_{\R}B$, set
$\omega^{H\overline{H}}(U,V)=\omega^M(U^H,V^H)$ where $\omega^M$
is the closed form in the definition of K\"{a}hler fibration.
Furthermore, for $u>0$, set
$N_u=N_V+\frac{i\omega^{H\overline{H}}}{u}$. $N_u$ is a section of
$(f^*\Lambda(T^*_{\R}B)\widehat{\otimes}{\rm
End}(\Lambda(T^{*(0,1)}f)\otimes\xi))^{\rm even}$.
\end{defn}

We now turn to the equivariant case. Let $G$ be a compact Lie
group, we shall assume that all complex manifolds and holomorphic
morphisms considered above are $G-$equivariant and all metrics are
$G$-invariant. We will additionally assume that the direct images
$R^kf_*\xi$ are all locally free so that the $G$-equivariant
coherent sheaf $R^\cdot f_*\xi$ is locally free and hence a
$G$-equivariant vector bundle over $B$. \cite{Ma1}, 1.2 gives a
$G$-invariant hermitian metric (the $L^2$-metric) $h^{R^\cdot
f_*\xi}$ on the vector bundle $R^\cdot f_*\xi$.

For $g\in G$, let $M_g=\{x\in M\mid g\cdot x=x\}$ and $B_g=\{b\in
B\mid g\cdot b=b\}$ be the fixed point submanifolds, then $f$
induces a holomorphic submersion $f_g: M_g\rightarrow B_g$. Let
$\Phi$ be the homomorphism $\alpha\mapsto (2i\pi)^{-{\rm
deg}\alpha/2}$ of $\Lambda^{\rm even}(T_{\R}^*B)$ into itself. We
put
\begin{displaymath}
{\rm ch}_g(R^\cdot f_*\xi,h^{R^\cdot f_*\xi})=\sum_{k=0}^{{\rm
dim}M-{\rm dim}B}(-1)^k{\rm ch}_g(R^k f_*\xi,h^{R^k f_*\xi})
\end{displaymath} and
\begin{displaymath}
{\rm ch}_g'(R^\cdot f_*\xi,h^{R^\cdot f_*\xi})=\sum_{k=0}^{{\rm
dim}M-{\rm dim}B}(-1)^kk{\rm ch}_g(R^k f_*\xi,h^{R^k f_*\xi}).
\end{displaymath}

\begin{defn}\label{315}
For $s\in \C$ with ${\rm Re}(s)>1$, let
\begin{displaymath}
\zeta_1(s)=-\frac{1}{\Gamma(s)}\int_0^1 u^{s-1}(\Phi{\rm
Tr_s}[gN_u{\rm exp}(-B_u^2)]-{\rm ch}_g'(R^\cdot f_*\xi,h^{R^\cdot
f_*\xi})){\rm d}u
\end{displaymath} and similarly for $s\in \C$
with ${\rm Re}(s)<\frac{1}{2}$, let
\begin{displaymath}
\zeta_2(s)=-\frac{1}{\Gamma(s)}\int_1^\infty u^{s-1}(\Phi{\rm
Tr_s}[gN_u{\rm exp}(-B_u^2)]-{\rm ch}_g'(R^\cdot f_*\xi,h^{R^\cdot
f_*\xi})){\rm d}u.
\end{displaymath}
\end{defn}

X. Ma proves that $\zeta_1(s)$ extends to a holomorphic function
of $s\in \C$ near $s=0$ and $\zeta_2(s)$ is a holomorphic function
of $s$.

\begin{defn}\label{316}
The smooth form $T_g(\omega^M,h^\xi):=\frac{\partial}{\partial
s}(\zeta_1+\zeta_2)(0)$ on $B_g$ is called the equivariant
analytic torsion form.
\end{defn}

\begin{thm}\label{317}
The form $T_g(\omega^M,h^\xi)$ lies in $\bigoplus_{p\geq
0}A^{p,p}(B_g)$ and satisfies the following differential equation
\begin{displaymath}
{\rm dd}^cT_g(\omega^M,h^\xi)={\rm ch}_g(R^\cdot f_*\xi,h^{R^\cdot
f_*\xi})-\int_{M_g/{B_g}}{\rm Td}_g(Tf,h^{Tf}){\rm
ch}_g(\xi,h^\xi).
\end{displaymath} Here $A^{p,p}(B_g)$ stands for
the space of smooth forms on $B_g$ of type $(p,p)$.
\end{thm}
\begin{proof}
This is \cite{Ma1}, Theorem 2.12.
\end{proof}

We define a secondary characteristic class
\begin{displaymath}
\widetilde{{\rm ch}}_g(R^\cdot f_*\xi,h^{R^\cdot
f_*\xi},h'^{R^\cdot f_*\xi}):=\sum_{k=0}^{{\rm dim}M-{\rm
dim}B}(-1)^k\widetilde{{\rm ch}}_g(R^k f_*\xi,h^{R^k
f_*\xi},h'^{R^k f_*\xi})
\end{displaymath}
such that it satisfies the following differential equation
\begin{displaymath}
{\rm dd}^c\widetilde{{\rm ch}}_g(R^\cdot f_*\xi,h^{R^\cdot
f_*\xi},h'^{R^\cdot f_*\xi})={\rm ch}_g(R^\cdot f_*\xi,h^{R^\cdot
f_*\xi})-{\rm ch}_g(R^\cdot f_*\xi,h'^{R^\cdot f_*\xi}),
\end{displaymath}
then the anomaly formula can be described as follows.

\begin{thm}\label{318}(Anomaly formula)
Let $\omega'$ be the form associated to another K\"{a}hler
fibration for $f: M\rightarrow B$. Let $h'^{Tf}$ be the metric on
$Tf$ in this new fibration and let $h'^\xi$ be another metric on
$\xi$. The following identity holds in
$\widetilde{A}(B_g):=\bigoplus_{p\geq 0}(A^{p,p}(B_g)/{({\rm
Im}\partial+{\rm Im}\overline{\partial})})$:
\begin{eqnarray*}
T_g(\omega^M,h^\xi)-T_g(\omega'^M,h'^\xi)&=&\widetilde{{\rm
ch}}_g(R^\cdot f_*\xi,h^{R^\cdot f_*\xi},h'^{R^\cdot
f_*\xi})-\int_{M_g/{B_g}}[\widetilde{{\rm
Td}}_g(Tf,h^{Tf},h'^{Tf}){\rm ch}_g(\xi,h^\xi)\\
&&+{\rm Td}_g(Tf,h'^{Tf})\widetilde{{\rm
ch}}_g(\xi,h^\xi,h'^\xi)].
\end{eqnarray*}
In particular, the class of $T_g(\omega^M,h^\xi)$ in
$\widetilde{A}(B_g)$ only depends on $(h^{Tf},h^\xi)$.
\end{thm}
\begin{proof}
This is \cite{Ma1}, Theorem 2.13.
\end{proof}

\subsection{Equivariant Bott-Chern singular currents}
The Bott-Chern singular current was defined by J.-M. Bismut, H.
Gillet and C. Soul\'{e} in \cite{BGS3} in order to generalize the
usual Bott-Chern secondary characteristic class to the case where
one considers the resolutions of hermitian vector bundles
associated to the closed immersions of complex manifolds. In
\cite{Bi}, J.-M. Bismut generalized this topic to the equivariant
case. We shall recall Bismut's construction of the equivariant
Bott-Chern singular current in this subsection since it plays a
crucial role in our later arguments. Bismut's construction was
realized via some current valued zeta function which involves the
supertraces of Quillen's superconnections. This is similar to the
non-equivariant case.

As before, let $g$ be the automorphism corresponding to an element
in a compact Lie group $G$. Let $i: Y\rightarrow X$ be an
equivariant closed immersion of $G$-equivariant K\"{a}hler
manifolds, and let $\overline{\eta}$ be an equivariant hermitian
vector bundle on $Y$. Assume that $\overline{\xi}.$ is a complex
of equivariant hermitian vector bundles on $X$ which provides a
resolution of $i_*\overline{\eta}$. We denote the differential of
the complex $\xi.$ by $v$. Note that $\xi.$ is acyclic outside $Y$
and the homology sheaves of its restriction to $Y$ are locally
free and hence they are all vector bundles. We write
$H_n=\mathcal{H}_n(\xi.\mid_Y)$ and define a $\Z$-graded bundle
$H=\bigoplus_nH_n$. For each $y\in Y$ and $u\in TX_y$, we denote
by $\partial_uv(y)$ the derivative of $v$ at $y$ in the direction
$u$ in any given holomorphic trivialization of $\xi.$ near $y$.
Then the map $\partial_uv(y)$ acts on $H_y$ as a chain map, and
this action only depends on the image $z$ of $u$ in $N_y$ where
$N$ stands for the normal bundle of $i(Y)$ in $X$. So we get a
chain complex of holomorphic vector bundles $(H,\partial_zv)$.

Let $\pi$ be the projection from the normal bundle $N$ to $Y$,
then we have a canonical identification of $\Z$-graded chain
complexes
\begin{displaymath}
(\pi^*H,\partial_zv)\cong(\pi^*(\wedge N^\vee\otimes\eta),i_z).
\end{displaymath}
For this, one can see \cite{Bi2}, Section I.b. Moreover, such an
identification is an identification of $G$-bundles which induces a
family of canonical isomorphisms $\gamma_n: H_n\cong\wedge^n
N^\vee\otimes\eta$. Another way to describe these canonical
isomorphisms $\gamma_n$ is applying \cite{GBI}, Exp. VII, Lemma
2.4 and Proposition 2.5. These two constructions coincide because
they are both locally, on a suitable open covering $\{U_j\}_{j\in
J}$, determined by any complex morphism over the identity map of
$\eta\mid_{U_j}$ from $(\xi.\mid_{U_j},v)$ to the minimal
resolution of $\eta\mid_{U_j}$ (e.g. the Koszul resolution). The
advantage of using the construction given in \cite{GBI} is that it
remains valid for arithmetic varieties over any base instead of
the complex numbers. Later in \cite{Bi}, for the use of
normalization, J.-M. Bismut considered the automorphism of
$N^\vee$ defined by multiplying a constant $-\sqrt{-1}$, it
induces an isomorphism of chain complexes
\begin{displaymath}
(\pi^*(\wedge N^\vee\otimes\eta),i_z)\cong (\pi^*(\wedge
N^\vee\otimes\eta),\sqrt{-1}i_z)
\end{displaymath}
and hence
\begin{displaymath}
(\pi^*H,\partial_zv)\cong(\pi^*(\wedge
N^\vee\otimes\eta),\sqrt{-1}i_z).
\end{displaymath}
This identification induces a family of isomorphisms
$\widetilde{\gamma_n}: H_n\cong\wedge^n N^\vee\otimes\eta$. By
finite dimensional Hodge theory, for each $y\in Y$, there is a
canonical isomorphism
\begin{displaymath}
H_y\cong\{f\in \xi._y\mid vf=0, v^*f=0\}
\end{displaymath}
where
$v^*$ is the dual of $v$ with respect to the metrics on $\xi.$.
This means that $H$ can be regarded as a smooth $\Z$-graded
$G$-equivariant subbundle of $\xi$ so that it carries an induced
$G$-invariant metric. On the other hand, we endow $\wedge
N^\vee\otimes \eta$ with the metric induced from $\overline{N}$
and $\overline{\eta}$. J.-M. Bismut introduced the following
definition.

\begin{defn}\label{321}
We say that the metrics on the complex of equivariant hermitian
vector bundles $\overline{\xi}.$ satisfy Bismut assumption (A) if
the identification $(\pi^*H,\partial_zv)\cong(\pi^*(\wedge
N^\vee\otimes\eta),\sqrt{-1}i_z)$ also identifies the metrics, it
is equivalent to the condition that the identification
$(\pi^*H,\partial_zv)\cong(\pi^*(\wedge N^\vee\otimes\eta),i_z)$
identifies the metrics.
\end{defn}

\begin{rem}\label{isometries}
If the metrics on the complex of equivariant hermitian vector
bundles $\overline{\xi}.$ satisfy Bismut assumption (A), then the
isomorphisms $\gamma_n$ and $\widetilde{\gamma_n}$ are all
isometries.
\end{rem}

\begin{prop}\label{322}
There always exist $G$-invariant metrics on $\xi.$ which satisfy
Bismut assumption (A) with respect to the equivariant hermitian
vector bundles $\overline{N}$ and $\overline{\eta}$.
\end{prop}
\begin{proof}
This is \cite{Bi}, Proposition 3.5.
\end{proof}

From now on we always suppose that the metrics on a resolution
satisfy Bismut assumption (A). Let $\nabla^{\xi}$ be the canonical
hermitian holomorphic connection on $\xi.$, then for each $u>0$,
we may define a $G$-invariant superconnection
\begin{displaymath}
C_u:=\nabla^\xi+\sqrt{u}(v+v^*)
\end{displaymath} on the
$\Z_2$-graded vector bundle $\xi$. Moreover, let $\Phi$ be the map
$\alpha\in \wedge(T_\R^*X_g)\rightarrow (2\pi i)^{-{\rm
deg}\alpha/2}\alpha\in \wedge(T_\R^*X_g)$ and denote
\begin{displaymath}
({\rm Td}_g^{-1})'(\overline{N}):=\frac{\partial}{\partial
b}\mid_{b=0}({\rm Td}_g(b\cdot {\rm
Id}-\frac{\Omega^{\overline{N}}}{2\pi i})^{-1})
\end{displaymath}
where $\Omega^{\overline{N}}$ is the curvature form associated to
$\overline{N}$. We recall as follows the construction of the
equivariant singular current given in \cite{Bi}, Section VI.

\begin{lem}\label{323}
Let $N_H$ be the number operator on the complex $\xi.$ (of
homological type), then for $s\in \C$ and $0< {\rm
Re}(s)<\frac{1}{2}$, the current valued zeta function
\begin{displaymath}
Z_g(\overline{\xi}.)(s):=\frac{1}{\Gamma(s)}\int_0^\infty
u^{s-1}[\Phi{\rm Tr_s}(N_Hg{\rm exp}(-C_u^2))+({\rm
Td}_g^{-1})'(\overline{N}){\rm
ch}_g(\overline{\eta})\delta_{Y_g}]{\rm d}u
\end{displaymath} is
well-defined on $X_g$ and it has a meromorphic continuation to the
complex plane which is holomorphic at $s=0$.
\end{lem}

\begin{defn}\label{324}
The equivariant singular Bott-Chern current on $X_g$ associated to
the resolution $\overline{\xi}.$ is defined as
\begin{displaymath}
T_g(\overline{\xi}.):=\frac{\partial}{\partial
s}\mid_{s=0}Z_g(\overline{\xi}.)(s).
\end{displaymath}
\end{defn}

\begin{thm}\label{325}
The current $T_g(\overline{\xi}.)$ is a sum of $(p,p)$-currents
and it satisfies the differential equation
\begin{displaymath}
{\rm dd}^cT_g(\overline{\xi}.)={i_g}_*{\rm
ch}_g(\overline{\eta}){\rm
Td}_g^{-1}(\overline{N})-\sum_k(-1)^k{\rm ch}_g(\overline{\xi}_k).
\end{displaymath} Moreover, the wave front set of
$T_g(\overline{\xi}.)$ is contained in $N^\vee_{g,\R}$.
\end{thm}
\begin{proof}
This follows from \cite{Bi}, Theorem 6.7 and Remark 6.8.
\end{proof}

Finally, we recall a theorem concerning the relationship of the
equivariant Bott-Chern singular currents involved in a double
complex. This theorem makes sure that our definition for an
embedding morphism in arithmetic $K_0$-theory is reasonable.

\begin{thm}\label{326}
Let \begin{displaymath} \overline{\chi}:\quad 0\rightarrow
\overline{\eta}_n\rightarrow\cdots\rightarrow\overline{\eta}_1\rightarrow\overline{\eta}_0\rightarrow0
\end{displaymath} be an exact sequence of equivariant hermitian
vector bundles on $Y$. Assume that we have the following double
complex consisting of resolutions of $i_*\overline{\chi}$ such
that all rows are exact sequences.
\begin{displaymath}
\xymatrix{0
\ar[r] & \overline{\xi}_{n,\cdot} \ar[r] \ar[d] & \cdots \ar[r] &
\overline{\xi}_{1,\cdot} \ar[r] \ar[d] &
\overline{\xi}_{0,\cdot} \ar[r] \ar[d] & 0 \\
0 \ar[r] & i_*\overline{\eta}_n \ar[r] & \cdots \ar[r] &
i_*\overline{\eta}_1 \ar[r] & i_*\overline{\eta}_0 \ar[r] & 0.}
\end{displaymath} For each $k$, we write $\overline{\varepsilon}_k$
for the exact sequence
\begin{displaymath} 0\rightarrow
\overline{\xi}_{n,k}\rightarrow \cdots\rightarrow
\overline{\xi}_{1,k}\rightarrow \overline{\xi}_{0,k}\rightarrow 0.
\end{displaymath} Then we have the following equality in
$\widetilde{\mathcal{U}}(X_g):=\bigoplus_{p\geq0}(D^{p,p}(X_g)/({\rm
Im}\partial+{\rm Im}\overline{\partial}))$
\begin{displaymath}
\sum_{j=0}^n(-1)^jT_g(\overline{\xi}_{j,\cdot})={i_g}_*\frac{\widetilde{{\rm
ch}}_g(\overline{\chi})}{{\rm
Td}_g(\overline{N})}-\sum_k(-1)^k\widetilde{{\rm
ch}}_g(\overline{\varepsilon}_k).
\end{displaymath} Here
$D^{p,p}(X_g)$ stands for the space of currents on $X_g$ of type
$(p,p)$.
\end{thm}
\begin{proof}
This is \cite{KR1}, Theorem 3.14.
\end{proof}

\subsection{Bismut-Ma's immersion formula}
In this subsection, we shall recall Bismut-Ma's immersion formula
which reflects the behaviour of the equivariant analytic torsion
forms of a K\"{a}hler fibration under the composition of an
immersion and a submersion. Such a formula can be used to measure,
in arithmetic $K_0$-theory, the difference between a push-forward
morphism and the composition formed as an embedding morphism
followed by a push-forward morphism.

Let $i: Y\rightarrow X$ be an equivariant closed immersion of
$G$-equivariant K\"{a}hler manifolds. Let $S$ be a complex
manifold with trivial $G$-action, and let $f: Y\rightarrow S$, $l:
X\rightarrow S$ be two equivariant proper holomorphic submersions
such that $f=l\circ i$. Assume that $\overline{\eta}$ is an
equivariant hermitian vector bundle on $Y$ and $\overline{\xi}.$
provides a resolution of $i_*\overline{\eta}$ on $X$ whose metrics
satisfy Bismut assumption (A). Let $\omega^Y$, $\omega^X$ be the
real, closed and $G$-invariant $(1,1)$-forms on $Y$, $X$ which
induce the K\"{a}hler fibrations with respect to $f$ and $l$
respectively. We shall assume that $\omega^Y$ is the pull-back of
$\omega^X$ so that the K\"{a}hler metric on $Y$ is induced by the
K\"{a}hler metric on $X$. As before, denote by $N$ the normal
bundle of $i(Y)$ in $X$. Consider the following exact sequence
\begin{displaymath}
\overline{\mathcal{N}}:\quad 0\rightarrow \overline{Tf}\rightarrow
\overline{Tl}\mid_Y\rightarrow \overline{N}\rightarrow 0
\end{displaymath} where $N$ is endowed with the quotient metric, we
shall write $\widetilde{{\rm
Td}}_g(\overline{Tf},\overline{Tl}\mid_Y)$ for $\widetilde{{\rm
Td}}_g(\overline{\mathcal{N}})$ the equivariant Todd secondary
characteristic class associated to $\overline{\mathcal{N}}$. It
satisfies the following differential equation
\begin{displaymath} {\rm dd}^c\widetilde{{\rm
Td}}_g(\overline{Tf},\overline{Tl}\mid_Y)={\rm
Td}_g(Tf,h^{Tf}){\rm Td}_g(\overline{N})-{\rm
Td}_g(Tl\mid_Y,h^{Tl}).
\end{displaymath}

For simplicity, we shall suppose that in the resolution $\xi.$,
$\xi_j$ are all $l-$acyclic and moreover $\eta$ is $f-$acyclic. By
an easy argument of long exact sequence, we have the following
exact sequence
\begin{displaymath} \Xi:\quad 0\rightarrow
l_*(\xi_m)\rightarrow l_*(\xi_{m-1})\rightarrow\ldots\rightarrow
l_*(\xi_0)\rightarrow f_*\eta\rightarrow 0.
\end{displaymath} By
the semi-continuity theorem, all the elements in the exact
sequence above are vector bundles. In this case, we recall the
definition of the $L^2$-metrics on direct images precisely as
follows. We just take $f_*h^\eta$ as an example. Note that the
semi-continuity theorem implies that the natural map
\begin{displaymath}
(R^0f_*\eta)_s\rightarrow H^0(Y_s,\eta\mid_{Y_s})
\end{displaymath} is an isomorphism for every point $s\in S$ where
$Y_s$ stands for the fibre over $s$. We may endow
$H^0(Y_s,\eta\mid_{Y_s})$ with a $L^2$-metric given by the formula
\begin{displaymath}
<u,v>_{L^2}:=\frac{1}{(2\pi)^{d_s}}\int_{Y_s}h^\eta(u,v)\frac{{\omega^Y}^{d_s}}{d_s!}
\end{displaymath} where $d_s$ is the complex dimension of the fibre
$Y_s$. It can be shown that these metrics depend on $s$ in a
$C^\infty$ manner (cf. \cite{BGV}, p.278) and hence define a
hermitian metric on $f_*\eta$. We shall denote it by $f_*h^\eta$.

In order to understand the statement of Bismut-Ma's immersion
formula, we still have to recall an important concept defined by
J.-M. Bismut, the equivariant $R$-genus. Let $W$ be a
$G$-equivariant complex manifold, and let $\overline{E}$ be an
equivariant hermitian vector bundle on $W$. For $\zeta\in S^1$ and
$s>1$ consider the zeta function
\begin{displaymath}
L(\zeta,s)=\sum_{k=1}^\infty\frac{\zeta^k}{k^s}
\end{displaymath}
and its meromorphic continuation to the whole complex plane.
Define the formal power series in $x$
\begin{displaymath}
\widetilde{R}(\zeta,x):=\sum_{n=0}^\infty\big(\frac{\partial
L}{\partial
s}(\zeta,-n)+L(\zeta,-n)\sum_{j=1}^n\frac{1}{2j}\big)\frac{x^n}{n!}.
\end{displaymath}

\begin{defn}\label{331}
The Bismut equivariant $R$-genus of an equivariant hermitian
vector bundle $\overline{E}$ with
$\overline{E}\mid_{X_g}=\sum_\zeta\overline{E}_\zeta$ is defined
as \begin{displaymath} R_g(\overline{E}):=\sum_{\zeta\in
S^1}\big({\rm
Tr}\widetilde{R}(\zeta,-\frac{\Omega^{\overline{E}_\zeta}}{2\pi
i})-{\rm
Tr}\widetilde{R}(1/\zeta,\frac{\Omega^{\overline{E}_\zeta}}{2\pi
i})\big) \end{displaymath} where $\Omega^{\overline{E}_\zeta}$ is
the curvature form associated to $\overline{E}_\zeta$. Actually,
the class of $R_g(\overline{E})$ in $\widetilde{A}(X_g)$ is
independent of the metric and we just write $R_g(E)$ for it.
Furthermore, the class $R_g(\cdot)$ is additive.
\end{defn}

\begin{thm}\label{332}(Immersion formula)
Let notations and assumptions be as above. Then the equality
\begin{eqnarray*}
\sum_{i=0}^m(-1)^iT_g(\omega^X,h^{\xi_i})-T_g(\omega^Y,h^\eta)+\widetilde{{\rm
ch}}_g(\Xi,h^{L^2})&=&\int_{X_g/S}{\rm
Td}_g(Tl,h^{Tl})T_g(\overline{\xi}.)\\
+\int_{Y_g/S}\frac{\widetilde{{\rm
Td}}_g(\overline{Tf},\overline{Tl}\mid_Y)}{{\rm
Td}_g(\overline{N})}{\rm ch}_g(\overline{\eta})&+&\int_{X_g/S}{\rm
Td}_g(Tl)R_g(Tl)\sum_{i=0}^m(-1)^i{\rm
ch}_g(\xi_i)\\
&-&\int_{Y_g/S}{\rm Td}_g(Tf)R_g(Tf){\rm ch}_g(\eta)
\end{eqnarray*} holds in $\widetilde{A}(S)$.
\end{thm}
\begin{proof}
This is the combination of \cite{BM}, Theorem 0.1 and 0.2, the
main theorems in that paper.
\end{proof}

\section{Arithmetic concentration theorem}
It is the aim of this section to prove an arithmetic concentration
theorem in Arakelov geometry. Let $X$ be a $\mu_n$-equivariant
arithmetic variety, we consider a special closed immersion $i:
X_{\mu_n}\hookrightarrow X$ where $X_{\mu_n}$ is the fixed point
subscheme of $X$. We first claim that the morphism $i$ induces a
well-defined group homomorphism $i_*$ between equivariant
arithmetic $K_0$-groups as in the algebraic case. To construct
$i_*$, some analytic datum, which is the equivariant Bott-Chern
singular current, should be involved. Precisely speaking, let
$\overline{\eta}$ be a $\mu_n$-equivariant hermitian vector bundle
on $X_{\mu_n}$ and let $\overline{\xi}.$ be a bounded complex of
$\mu_n$-equivariant hermitian vector bundles which provides a
resolution of $i_*\overline{\eta}$ on $X$, then we may have an
equivariant Bott-Chern singular current $T_g(\overline{\xi}.)\in
\widetilde{\mathcal{U}}(X_{\mu_n})$. Note that the $0$-degree part
of the normal bundle $N:=N_{X/{X_{\mu_n}}}$ vanishes (cf.
\cite{KR1}, Prop. 2.12) so that the wave front set of
$T_g(\overline{\xi}.)$ is the empty set, and hence we know that
the current $T_g(\overline{\xi}.)$ is actually smooth. This fact
means that the following definition does make sense.

\begin{defn}\label{401}
Let notations and assumptions be as above. The embedding morphism
\begin{displaymath}
i_*:\quad \widehat{K_0}(X_{\mu_n},\mu_n)\rightarrow
\widehat{K_0}(X,\mu_n)
\end{displaymath} is defined as follows.

(i). For every $\mu_n$-equivariant hermitian vector bundle
$\overline{\eta}$ on $X_{\mu_n}$, suppose that $\overline{\xi}.$
is a resolution of $i_*\overline{\eta}$ on $X$ whose metrics
satisfy Bismut assumption (A),
$i_*[\overline{\eta}]=\sum_k(-1)^k[\overline{\xi}_k]+T_g(\overline{\xi.})$.

(ii). For every $\alpha\in \widetilde{A}(X_{\mu_n})$,
$i_*\alpha=\alpha{\rm Td}_g^{-1}(\overline{N})$.
\end{defn}

\begin{thm}\label{402}
The embedding morphism $i_*$ is a well-defined group homomorphism.
\end{thm}
\begin{proof}
We have to prove that our definition for $i_*$ is well-defined and
it is compatible with the two generating relations of arithmetic
$K_0$-group. Indeed, assume that we are given a short exact
sequence \begin{displaymath} \overline{\chi}:\quad 0\rightarrow
\overline{\eta}'\rightarrow \overline{\eta}\rightarrow
\overline{\eta}''\rightarrow 0
\end{displaymath} of equivariant
hermitian vector bundles on $X_{\mu_n}$. As in Theorem~\ref{326},
let $\overline{\xi}.'$, $\overline{\xi}.$ and $\overline{\xi}.''$
be resolutions on $X$ of $i_*\overline{\eta}'$,
$i_*\overline{\eta}$ and $i_*\overline{\eta}''$ which fit the
following double complex
\begin{displaymath}
\xymatrix{0 \ar[r] & \overline{\xi}.' \ar[r] \ar[d] &
\overline{\xi}. \ar[r] \ar[d] &
\overline{\xi}.'' \ar[r] \ar[d] & 0 \\
0 \ar[r] & i_*\overline{\eta}' \ar[r] & i_*\overline{\eta} \ar[r]
& i_*\overline{\eta}'' \ar[r] & 0}
\end{displaymath} such that all
rows are exact. For each $k$, we write $\overline{\varepsilon}_k$
for the exact sequence \begin{displaymath} 0\rightarrow
\overline{\xi}'_k\rightarrow \overline{\xi}_k\rightarrow
\overline{\xi}''_k\rightarrow 0. \end{displaymath} Then
Theorem~\ref{326} implies that the equality
\begin{displaymath}
T_g(\overline{\xi}.')-T_g(\overline{\xi}.)+T_g(\overline{\xi}.'')=\frac{\widetilde{{\rm
ch}}_g(\overline{\chi})}{{\rm
Td}_g(\overline{N})}-\sum_k(-1)^k\widetilde{{\rm
ch}}_g(\overline{\varepsilon}_k) \end{displaymath} holds in
$\widetilde{A}(X_{\mu_n})$. This means
$i_*[\overline{\eta}']-i_*[\overline{\eta}]+i_*[\overline{\eta}'']=0$
in the group $\widehat{K_0}(X,\mu_n)$ according to its generating
relations. Note that if $\overline{\xi.}$ is an exact sequence
then $T_g(\overline{\xi.})$ is equal to $-\widetilde{{\rm
ch}}_g(\overline{\xi.})$, so we have $i_*[0]=0$. Moreover, any two
resolutions of $i_*\overline{\eta}$ are dominated by a third one,
then our arguments above also show that $i_*[\overline{\eta}]$ is
independent of the choice of resolution. Therefore, the embedding
morphism $i_*$ is well-defined and it is compatible with the first
generating relation of arithmetic $K_0$-group. On the other hand,
the compatibility with the second relation is trivial. So we are
done.
\end{proof}

\begin{lem}\label{pf}(Projection formula)
For any elements $x\in \widehat{K_0}(X,\mu_n)$ and $y\in
\widehat{K_0}(X_{\mu_n},\mu_n)$, the equality $i_*(i^*x\cdot
y)=x\cdot i_*y$ holds in $\widehat{K_0}(X,\mu_n)$.
\end{lem}
\begin{proof}
Assume that $x=\overline{E}$ and $y=\overline{F}$ are equivariant
hermitian vector bundles. Let $\overline{\xi.}$ be a resolution of
$i_*\overline{F}$ on $X$, then $\overline{E}\otimes
\overline{\xi.}$ provides a resolution of
$i_*(i^*\overline{E}\otimes \overline{F})$. By definition we have
\begin{displaymath}
i_*(i^*x\cdot y)=\sum(-1)^k[\overline{\xi} _k\otimes
\overline{E}]+{\rm ch}_g(\overline{E})T_g(\overline{\xi.})
\end{displaymath} which is exactly $x\cdot i_*y$. Assume that
$x=\alpha$ is represented by some smooth form and $y=\overline{F}$
is an equivariant hermitian vector bundle. Again let
$\overline{\xi.}$ be a resolution of $i_*\overline{F}$ on $X$,
then \begin{displaymath} i_*(i^*x\cdot y)=\alpha{\rm
Td}_g^{-1}(\overline{N}){\rm ch}_g(\overline{F})=\alpha[{\rm
dd}^cT_g(\overline{\xi.})+\sum(-1)^k{\rm ch}_g(\overline{\xi}_k)]
\end{displaymath} which is exactly $x\cdot i_*y$. Now assume that
$x=\overline{E}$ is an equivariant hermitian vector bundle and
$y=\alpha$ is represented by some smooth form, then
\begin{displaymath}
i_*(i^*x\cdot y)=i_*({\rm ch}_g(\overline{E})\alpha)={\rm
ch}_g(\overline{E})\alpha{\rm Td}_g^{-1}(\overline{N})
\end{displaymath} which is exactly $x\cdot i_*y$. Finally, if $x$
and $y$ are both represented by smooth forms then $i_*(i^*x\cdot
y)$ is trivially equal to $x\cdot i_*y$ by definition. Note that
$i_*$ and $i^*$ are group homomorphisms, so we may conclude the
projection formula from its correctness on generators. This
completes the proof.
\end{proof}

\begin{rem}\label{pfrem}
Lemma~\ref{pf} implies that $i_*$ is even a homomorphism of
$R(\mu_n)$-modules so that it induces a homomorphism between
arithmetic $K_0$-groups after taking localization, and hence there
exists a corresponding projection formula after taking
localization.
\end{rem}

With Remark~\ref{pfrem}, the arithmetic concentration theorem can
be described as follows.

\begin{thm}\label{403}
The embedding morphism $i_*:
\widehat{K_0}(X_{\mu_n},\mu_n)_{\rho}\rightarrow
\widehat{K_0}(X,\mu_n)_{\rho}$ is an isomorphism and the inverse
morphism of $i_*$ is given by
$\lambda_{-1}^{-1}(\overline{N}^\vee)\cdot i^*$ where $N$ again
stands for the normal bundle $N_{X/{X_{\mu_n}}}$.
\end{thm}

Before giving a complete proof of this concentration theorem, we
need to make some purely analytic preliminaries.

\begin{defn}\label{404}
Let $X$ be a compact complex manifold, and let $\overline{\xi}.$
be a bounded complex of hermitian vector bundles on $X$ such that
the homology sheaves are all locally free i.e. vector bundles.
Suppose that the homology sheaves are endowed with some hermitian
metrics $h^H$. Such a complex will be called a standard complex.
Endow the kernel and the image of every differential with the
induced metrics from $\overline{\xi}.$. We say that a standard
complex $(\overline{\xi}.,h^H)$ is homologically split if the
following short exact sequences
\begin{displaymath}
0\rightarrow \overline{{\rm Im}}\rightarrow \overline{{\rm
Ker}}\rightarrow \overline{H}_*\rightarrow 0 \end{displaymath} and
\begin{displaymath}
0\rightarrow \overline{{\rm Ker}}\rightarrow
\overline{\xi}_*\rightarrow \overline{{\rm Im}}\rightarrow 0
\end{displaymath} of hermitian
vector bundles are all orthogonally split.
\end{defn}

In \cite{Ma2}, X. Ma proved the following uniqueness theorem.

\begin{thm}\label{405}
Let $X$ be a compact complex manifold, then to each standard
complex of hermitian vector bundles $(\overline{\xi}.,h^H)$ on $X$
there is a unique way to associate an element
$M(\overline{\xi}.,h^H)\in \widetilde{A}(X)$ satisfying the
following conditions.

(i). ${\rm dd}^cM(\overline{\xi}.,h^H)=\sum(-1)^i{\rm
ch}(\overline{H}_i)-\sum(-1)^j{\rm ch}(\overline{\xi}_j)$.

(ii). For any holomorphic morphism $f: X'\rightarrow X$, we have
$M(f^*\overline{\xi}.f^*h^H)=f^*M(\overline{\xi}.,h^H)$.

(iii). If $(\overline{\xi}.,h^H)$ is homologically split, then
$M(\overline{\xi}.,h^H)=0$.
\end{thm}

The definition of standard complex and Ma's uniqueness theorem can
be easily generalized to the equivariant case. We summarize these
generalizations as follows.

\begin{defn}\label{thesis-311}
Let $X$ be a compact complex manifold which admits a holomorphic
action of a compact Lie group $G$. Fix an element $g\in G$. An
equivariant standard complex on $X$ is a bounded complex of
$G$-equivariant hermitian vector bundles on $X$ whose restriction
to $X_g$ is standard and the metrics on the homology sheaves are
$g$-invariant. Again we shall write an equivariant standard
complex as $(\overline{\xi}.,h^H)$ to emphasize the choice of the
metrics on the homology sheaves.
\end{defn}

\begin{thm}\label{406}
Let $X$ be a compact complex manifold which admits a holomorphic
action of a compact Lie group $G$. Fix an element $g\in G$. Then
to each equivariant standard complex $(\overline{\xi}.,h^H)$ on
$X$, there is a unique additive way to associate an element
$M_g(\overline{\xi}.,h^H)\in \widetilde{A}(X_g)$ satisfying the
following conditions.

(i). ${\rm dd}^cM_g(\overline{\xi}.,h^H)=\sum(-1)^i{\rm
ch}_g(\overline{H_i}(\xi.\mid_{X_g}))-\sum(-1)^j{\rm
ch}_g(\overline{\xi}_j)$.

(ii). For any equivariant holomorphic morphism $f: X'\rightarrow
X$, we have
$M_g(f^*\overline{\xi}.,f^*h^H)=f_g^*M_g(\overline{\xi}.,h^H)$.

(iii). If $(\overline{\xi}.\mid_{X_g},h^H)$ is homologically
split, then $M_g(\overline{\xi}.,h^H)=0$.
\end{thm}
\begin{proof}
The complex $\overline{\xi}.$ splits on $X_g$ orthogonally into a
series of standard complexes $\overline{\xi}_\zeta.$ for all
$\zeta\in S^1$. Using the non-equivariant Bott-Chern-Ma classes on
$X_g$, we define
\begin{displaymath}
M_g(\overline{\xi}.,h^H)=\sum_{\zeta\in S^1}\zeta
M(\overline{\xi}_\zeta.,h^{H_\zeta}). \end{displaymath} Then the
axiomatic characterization follows from the non-equivariant one in
Theorem~\ref{405} and the definition of ${\rm ch}_g$. For the
uniqueness, first note that by the condition (ii), the relation
$M_g(\overline{\xi}.,h^H)=M_g(\overline{\xi}.\mid_{X_g},h^H)$
should be satisfied, then we may reduce our proof to the case
where $X$ is equal to $X_g$. Since $M_g$ is required to be
additive, we only have to show that for every $\zeta\in S^1$,
$M_g(\overline{\xi}_\zeta.,h^{H_\zeta})=\zeta
M(\overline{\xi}_\zeta.,h^{H_\zeta})$. This follows from
Theorem~\ref{405} since every compact complex manifold can be
regarded as an equivariant complex manifold (with the trivial
action), on which any standard complex can be endowed with a
$g$-structure as multiplication by $\zeta$. Such an approach is
similar to the proof of \cite{KR1}, Theorem 3.4.
\end{proof}

\begin{rem}\label{thesis-rem}
(i). The condition of compactness in Definition~\ref{404} and
Theorem~\ref{405} is not necessary, it was only used in the proof
of Theorem~\ref{406} given above.

(ii). If one directly generalizes the proof of Theorem~\ref{405}
to the equivariant case (by trivially adding the subscript $g$ to
every notation), then the condition of additivity in
Theorem~\ref{406} can be removed. Actually the additivity is a
byproduct of such a proof.
\end{rem}

Now let $\overline{\xi}.$ be an equivariant standard complex on
$X$. Then we can always split $\overline{\xi}.\mid_{X_g}$ into a
series of short exact sequences
\begin{displaymath} 0\rightarrow \overline{{\rm Im}}\rightarrow
\overline{{\rm Ker}}\rightarrow
\overline{H}_*(\xi.\mid_{X_g})\rightarrow 0
\end{displaymath} and
\begin{displaymath}
0\rightarrow \overline{{\rm Ker}}\rightarrow
\overline{\xi}_*\mid_{X_g}\rightarrow \overline{{\rm
Im}}\rightarrow 0 \end{displaymath} of equivariant hermitian
vector bundles. Denote the alternating sum of the equivariant
Bott-Chern secondary characteristic classes of the short exact
sequences above by $C(\overline{\xi}.,h^H)$ such that it satisfies
the following differential equation \begin{displaymath} {\rm
dd}^cC(\overline{\xi}.,h^H)=\sum(-1)^i{\rm
ch}_g(\overline{H}_i(\xi.\mid_{X_g}))-\sum(-1)^j{\rm
ch}_g(\overline{\xi}_j), \end{displaymath} then it is clear that
the class $C(\overline{\xi}.,h^H)$ is an element in
$\widetilde{A}(X_g)$ and it satisfies the three conditions in
Theorem~\ref{406}. To every equivariant standard complex
$\overline{\xi.}$ on $X$, we may associated a new canonical
equivariant standard complex in which the metrics on homology
bundles $H_*(\xi.\mid_{X_g})$ are induced by the metrics on $\xi.$
(see Section 3.2). This special choice of metrics will be denoted
by $h^H_{\rm ind}$. It is easy to compute the difference of
$C(\overline{\xi}.,h^H)$ and $C(\overline{\xi}.,h^H_{\rm ind})$.
It is the alternating sum of secondary characteristic classes
\begin{displaymath}
\sum(-1)^i\widetilde{{\rm ch}}_g(H_i(\xi.\mid_{X_g}),h^H,h^H_{\rm
ind}). \end{displaymath}

Now we define another equivariant secondary class associated to
$(\overline{\xi}.,h^H_{\rm ind})$ by using the supertraces of
Quillen's superconnections as follows.

For $s\in \mathbb{C}$ with ${\rm Re}(s)>1$, let
\begin{displaymath}
\zeta_1(s)=-\frac{1}{\Gamma(s)}\int_0^1 u^{s-1}\{\Phi{\rm
Tr_s}[Ng{\rm exp}(-A_u^2)]-\Phi{\rm Tr_s}[Ng{\rm
exp}(-\nabla^{H(\overline{\xi}.),2})]\}{\rm d}u \end{displaymath}
and similarly for $s\in \mathbb{C}$ with ${\rm
Re}(s)<\frac{1}{2}$, let
\begin{displaymath}
\zeta_2(s)=-\frac{1}{\Gamma(s)}\int_1^\infty u^{s-1}\{\Phi{\rm
Tr_s}[Ng{\rm exp}(-A_u^2)]-\Phi{\rm Tr_s}[Ng{\rm
exp}(-\nabla^{H(\overline{\xi}.),2})]\}{\rm d}u. \end{displaymath}

We define $\zeta(\overline{\xi}.,h^H_{\rm
ind}):=\frac{\partial}{\partial s}(\zeta_1+\zeta_2)(0)$. This is
just a generalization of \cite{Ma2}, D\'{e}finition 10.3 in the
equivariant case, we refer to that paper for the explanation of
the notations appearing in the definition of the zeta-functions
above. We thank X. Ma for his comment that an argument similar to
the one in \cite{BGS1}, Cor. 1.30 can be used to prove the
following lemma.

\begin{lem}\label{ts}
Define $\zeta(\overline{\xi}.,h^H):=\zeta(\overline{\xi}.,h^H_{\rm
ind})+\sum(-1)^i\widetilde{{\rm
ch}}_g(H_i(\xi.\mid_{X_g}),h^H,h^H_{\rm ind})$. Then
$\zeta(\overline{\xi}.,h^H)$ determines an element in
$\widetilde{A}(X_g)$ which satisfies the three conditions in
Theorem~\ref{406} and hence we have
$\zeta(\overline{\xi}.,h^H)=C(\overline{\xi}.,h^H)$.
\end{lem}
\begin{proof}
Actually, according to \cite{Ma2}, Proposition 10.4, one just need
to add the subscript $g$ to every step in the argument given in
\cite{BGS1}, Cor. 1.30 and nothing else should be changed. Here,
we roughly describe that why $\zeta(\overline{\xi}.,h^H_{\rm
ind})=0$ when $(\overline{\xi}.\mid_{X_g},h^H)$ is homologically
split (note that in this case $h^H$ should be equal to $h^H_{\rm
ind}$). This can be seen from the following argument. If
$(\overline{\xi}.\mid_{X_g},h^H)$ is homologically split, then up
to isometries we may write
$\overline{E}_k:=\overline{\xi}_k\mid_{X_g}\cong
\overline{F}_k\oplus \overline{H}_k\oplus \overline{F}_{k-1}$
where $\{\overline{F}_k\}$ is a family of hermitian vector bundles
on $X_g$. Moreover, the differential $v$ is given by
$(v_1,v_2,v_3)\mapsto (v_3,0,0)$. So we compute directly that
$A_u^2\mid_{\overline{E}_k}=\nabla^2+u({\rm
Id}_{\overline{F}_k}\oplus {\rm Id}_{\overline{F}_{k-1}})$. This
equality implies that
\begin{eqnarray*}
{\rm Tr_s}[Ng{\rm exp}(-A_u^2)]&=&\sum_k(-1)^k\{{\rm
Tr}\mid_{\overline{F}_k}[kg{\rm exp}(-\nabla^2-u{\rm
Id})]\\
&&+{\rm Tr}\mid_{\overline{H}_k}[kg{\rm exp}(-\nabla^2)]+{\rm
Tr}\mid_{\overline{F}_{k-1}}[kg{\rm exp}(-\nabla^2-u{\rm Id})]\}
\end{eqnarray*}
and hence
\begin{displaymath}
\zeta_1(s)+\zeta_2(s)=-\frac{1}{\Gamma(s)}\int_0^\infty
u^{s-1}e^{-u}\{\Phi{\rm Tr_s}[Ng{\rm exp}(-\nabla^2)]\}{\rm
d}u=-\Phi{\rm Tr_s}[Ng{\rm exp}(-\nabla^2)] \end{displaymath}
which has nothing to do with $s$. So we get
$\zeta(\overline{\xi}.,h^H)=\frac{\partial}{\partial
s}(\zeta_1+\zeta_2)(0)=0$.
\end{proof}

\begin{cor}\label{tss}
We have $\zeta(\overline{\xi}.,h^H_{\rm
ind})=C(\overline{\xi}.,h^H_{\rm ind})$ in $\widetilde{A}(X_g)$.
\end{cor}

Now we go back to the arithmetic case. We consider the closed
immersion $i: X_{\mu_n}\hookrightarrow X$ with hermitian normal
bundle $\overline{N}$, and as before let $\overline{\eta}$ be an
equivariant hermitian vector bundle on $X_{\mu_n}$. Assume that
the complex $\overline{\xi.}$ provides a resolution of
$i_*\overline{\eta}$ on $X$ by equivariant hermitian vector
bundles whose metrics satisfy Bismut assumption (A). Then the
restriction of $\overline{\xi}_\C.$ to $X_g$ is naturally a
standard complex such that $h^H$ is equal to $h^H_{\rm ind}$.

\begin{lem}\label{407}
Let notations and assumptions be as above. Then the equality
\begin{displaymath}
\lambda_{-1}(\overline{N}^\vee)\cdot
\overline{\eta}-\sum_j(-1)^ji^*(\overline{\xi}_j)=T_g(\overline{\xi}.)
\end{displaymath} holds in $\widehat{K_0}(X_{\mu_n},\mu_n)$.
\end{lem}
\begin{proof}
According to Remark~\ref{isometries}, we can split
$\overline{\xi.}\mid_{X_{\mu_n}}$ into the following series of
exact sequences of equivariant hermitian vector bundles
\begin{displaymath}
0\rightarrow \overline{{\rm Im}}\rightarrow \overline{{\rm
Ker}}\rightarrow
\wedge^*\overline{N}^\vee\otimes\overline{\eta}\rightarrow 0
\end{displaymath} and
\begin{displaymath}
0\rightarrow \overline{{\rm Ker}}\rightarrow
\overline{\xi}_*\mid_{X_{\mu_n}}\rightarrow \overline{{\rm
Im}}\rightarrow 0. \end{displaymath} Then by the definition of
arithmetic $K_0$-theory, $\lambda_{-1}(\overline{N}^\vee)\cdot
\overline{\eta}-\sum_j(-1)^ji^*(\overline{\xi}.)$ is nothing but
$C(\overline{\xi}.,h^H)$ or $C(\overline{\xi}.,h^H_{\rm ind})$ in
$\widehat{K_0}(X_{\mu_n},\mu_n)$.

Comparing with the construction of the equivariant Bott-Chern
singular current recalled in Section 3.2 or with more details in
\cite{Bi}, Section VI, we claim that in our special situation
$T_g(\overline{\xi}.)$ is equal to $\zeta(\overline{\xi}.,h^H_{\rm
ind})$ defined before this lemma. Actually since $\xi.$ is
supposed to admit the metrics satisfying Bismut assumption (A),
the superconnection $A_u$ in the definition of
$\zeta(\overline{\xi}.,h^H_{\rm ind})$ is exactly the
superconnection $C_u$ in the definition of $T_g(\overline{\xi}.)$.
Moreover, since $(H_*(\xi.\mid_{X_{\mu_n}}),h^H_{\rm ind})$ are
isometric to
$\wedge^*\overline{N}_\C^\vee\otimes\overline{\eta}_\C$ the
supertrace ${\rm Tr_s}[Ng{\rm
exp}(-\nabla^{H(\overline{\xi}.),2})]$ in the definition of
$\zeta(\overline{\xi}.,h^H_{\rm ind})$ is equal to $-({\rm
Td}_g^{-1})'(\overline{N}){\rm ch}_g(\overline{\eta})$ in the
definition of $T_g(\overline{\xi}.)$, this can be seen directly
from the computation \cite{Bi}, (6.26). So according to
Corollary~\ref{tss} we have $C(\overline{\xi}.,h^H_{\rm
ind})=\zeta(\overline{\xi}.,h^H_{\rm ind})=T_g(\overline{\xi}.)$
in $\widetilde{A}(X_{\mu_n})$ and hence they are equal in the
group $\widehat{K_0}(X_{\mu_n},\mu_n)$. This implies the equality
in the statement of this lemma.
\end{proof}

\begin{rem}\label{additional1}
The kernel of the proof of Lemma~\ref{407} is the equality
$C(\overline{\xi}.,h^H_{\rm ind})=T_g(\overline{\xi}.)$ which is
in the purely analytic setting. We would like to indicate that it
is possible to prove this equality by using the construction of
the deformation to the normal cone. It is a little complicated but
it is independently interesting. We shall fulfill this work in a
coming paper.
\end{rem}

We are now ready to give a complete proof of our arithmetic
concentration theorem.

\begin{proof}(of Theorem~\ref{403})
Denote by $U$ the complement of $X_{\mu_n}$ in $X$, then $j:
U\hookrightarrow X$ is an $\mu_n$-equivariant open subscheme of
$X$ whose fixed point set is empty. We consider the following
double complex
\begin{displaymath}
\xymatrix{\widetilde{A}(X_{\mu_n})_{\rho} \ar[r]^-{i_*} \ar[d]^{a}
& \widetilde{A}(X_{\mu_n})_{\rho} \ar[r]^-{j^*} \ar[d]^{a} &
\widetilde{A}(U_{\mu_n})_{\rho} \ar[r] \ar[d]^{a} & 0 \\
\widehat{K_0}(X_{\mu_n},\mu_n)_{\rho} \ar[r]^-{i_*} \ar[d]^\pi &
\widehat{K_0}(X,\mu_n)_{\rho} \ar[r]^-{j^*} \ar[d]^\pi &
\widehat{K_0}(U,\mu_n)_{\rho} \ar[r] \ar[d]^\pi & 0 \\
K_0(X_{\mu_n},\mu_n)_{\rho} \ar[r]^-{i_*} \ar[d] &
K_0(X,\mu_n)_{\rho} \ar[r]^-{j^*} \ar[d] & K_0(U,\mu_n)_{\rho}
\ar[r] \ar[d] & 0 \\
0 & 0 & 0 & }
\end{displaymath} whose columns are all exact
according to Remark~\ref{204}. The third row of this complex is
exact by Quillen's localization sequence for higher equivariant
$K$-theory. We next claim that the other two rows are also exact.
Actually, this follows from an easy argument of diagram chasing.
Note that the algebraic concentration theorem implies that
$K_*(U,\mu_n)_{\rho}$ all vanish, then together with the fact that
$U_{\mu_n}$ is the empty set we conclude that the elements of the
third column of this double complex are all equal to $0$. So our
claim is equivalent to say $i_*:
\widehat{K_0}(X_{\mu_n},\mu_n)_{\rho}\rightarrow
\widehat{K_0}(X,\mu_n)_{\rho}$ is surjective. Indeed, for any
element $x\in \widehat{K_0}(X,\mu_n)_{\rho}$ we may find an
element $y\in \widehat{K_0}(X_{\mu_n},\mu_n)_{\rho}$ such that
$i_*\pi(y)=\pi(x)$. This means $x-i_*(y)$ is in the kernel of
$\pi$, so there exists an element $\alpha\in
\widetilde{A}(X_{\mu_n})$ such that $\alpha=x-i_*(y)$ in
$\widehat{K_0}(X,\mu_n)_{\rho}$. Set $\beta=\alpha{\rm
Td}_g(\overline{N})$, we get $i_*(y+\beta)=i_*(y)+\alpha=x$ in
$\widehat{K_0}(X,\mu_n)_{\rho}$. Hence, $i_*$ is surjective. By
constructing its inverse morphism, we conclude that the embedding
morphism $i_*: \widehat{K_0}(X_{\mu_n},\mu_n)_{\rho}\rightarrow
\widehat{K_0}(X,\mu_n)_{\rho}$ is really an isomorphism.
Concerning such inverse morphism, let $\overline{\eta}$ be an
equivariant hermitian vector bundle on $X_{\mu_n}$, then we have
\begin{displaymath}
\lambda_{-1}^{-1}(\overline{N}^\vee)\cdot
i^*i_*([\overline{\eta}])=\lambda_{-1}^{-1}(\overline{N}^\vee)\cdot
i^*(\sum_k(-1)^k[\overline{\xi}_k]+T_g(\overline{\xi.}))=[\overline{\eta}].
\end{displaymath} The last equality follows from Lemma~\ref{407}.
Moreover, let $\alpha$ be an element in
$\widetilde{A}(X_{\mu_n})$, then we have the equalities
\begin{displaymath}
\lambda_{-1}^{-1}(\overline{N}^\vee)\cdot
i^*i_*([\alpha])=\lambda_{-1}^{-1}(\overline{N}^\vee)\cdot
[\alpha{\rm Td}_g^{-1}(\overline{N})]=[{\rm
ch}_g(\lambda_{-1}^{-1}(\overline{N}^\vee))\alpha{\rm
Td}_g^{-1}(\overline{N})]=[\alpha].
\end{displaymath} Therefore,
the inverse morphism of the embedding morphism $i_*$ is of the
form $\lambda_{-1}^{-1}(\overline{N}^\vee)\cdot i^*$. This
completes the proof of the arithmetic concentration theorem.
\end{proof}

\section{Relative fixed point formula of Lefschetz type in Arakelov geometry}
Let $X$, $Y$ be two $\mu_n$-equivariant arithmetic varieties over
some fixed arithmetic ring $D$ and let $f: X\rightarrow Y$ be a
$\mu_n$-equivariant morphism over $D$ which is smooth over the
complex numbers. By definition $X$ admits a $\mu_n$-projective
action over $D$, then the morphism $f$ is automatically
$\mu_n$-projective. Fix a $\mu_n(\C)$-invariant K\"{a}hler metric
on $X(\C)$ so that we get a K\"{a}hler fibration with respect to
the holomorphic submersion $f_\C: X(\C)\rightarrow Y(\C)$. Let
$\overline{E}$ be an $f$-acyclic $\mu_n$-equivariant hermitian
vector bundle on $X$, we know that the direct image $f_*E$ is a
coherent sheaf which is locally free on $Y(\C)$ hence it can be
endowed with a natural equivariant structure and the $L^2$-metric.
Moreover, $(f_*E,f_*h^E)$ always admits a resolution by
equivariant hermitian vector bundles on $Y$. Assume that
${\overline{\xi.}}^{\overline{E}}$ is such a resolution and we
denote by $\widetilde{{\rm
ch}}_g({\overline{\xi.}}^{\overline{E}})$ the secondary Bott-Chern
characteristic class of the following exact sequence
\begin{displaymath}
0\to {\overline{\xi}_m^{\overline{E}}}_\C\to \cdots\to
{\overline{\xi}_1^{\overline{E}}}_\C \to
{\overline{\xi}_0^{\overline{E}}}_\C \to (f_*E_\C,f_*h^E)\to 0.
\end{displaymath} Let $\widehat{K_0}^{\rm ac}(X,\mu_n)$ be the
group generated by $f$-acyclic equivariant vector bundles on $X$
and the elements of $\widetilde{A}(X_{\mu_n})$, with the same
relations as in Definition~\ref{203}. A theorem of Quillen (cf.
\cite{Qu1}, Cor.3 p. 111) for the algebraic analogs of this group
implies that the natural map $\widehat{K_0}^{\rm ac}(X,\mu_n)\to
\widehat{K_0}(X,\mu_n)$ is an isomorphism. So the following
definition does make sense.

\begin{defn}\label{501}
Let notations and assumptions be as above. The push-forward
morphism \begin{displaymath} f_*:\quad
\widehat{K_0}(X,\mu_n)\rightarrow \widehat{K_0}(Y,\mu_n)
\end{displaymath} is defined as follows.

(i). For every $f$-acyclic $\mu_n$-equivariant hermitian vector
bundle $\overline{E}$ on $X$, the direct image of $\overline{E}$
is given by
$f_*\overline{E}=\sum_k(-1)^k\overline{\xi}_k^{\overline{E}}+\widetilde{{\rm
ch}}_g({\overline{\xi.}}^{\overline{E}})-T_g(\omega^X,h^E)$ where
${\overline{\xi.}}^{\overline{E}}$ is a resolution of
$(f_*E,f_*h^E)$ on the base variety $Y$.

(ii). For every element $\alpha\in \widetilde{A}(X_{\mu_n})$,
$f_*\alpha=\int_{X_g/{Y_g}}{\rm Td}_g(Tf,h^{Tf})\alpha\in
\widetilde{A}(Y_{\mu_n})$.
\end{defn}

\begin{thm}\label{502}
The push-forward morphism $f_*$ is a well-defined group
homomorphism.
\end{thm}
\begin{proof}
Again we have to prove that our definition for $f_*$ is
independent of the choice of resolution and it is compatible with
the two generating relations of arithmetic $K_0$-group. Indeed,
assume that we are given a short exact sequence
\begin{displaymath}
\overline{\varepsilon}:\quad
0\rightarrow\overline{E}'\rightarrow\overline{E}\rightarrow\overline{E}''\rightarrow
0 \end{displaymath} of $f$-acyclic equivariant hermitian vector
bundles on $X$. Then we may construct the following short exact
sequence of resolution \begin{displaymath} \xymatrix{ &
\overline{\varepsilon}_m & \cdots & \overline{\varepsilon}_1 &
\overline{\varepsilon}_0 &
f_*\overline{\varepsilon} &  \\
 & 0 \ar[d] & & 0 \ar[d] & 0 \ar[d] & 0 \ar[d] & \\
0 \ar[r] & \overline{\xi}_m^{\overline{E}'}  \ar[r] \ar[d] &
\cdots \ar[r] & \overline{\xi}_1^{\overline{E}'} \ar[r] \ar[d] &
\overline{\xi}_0^{\overline{E}'} \ar[r] \ar[d] & (f_*E',f_*h^{E'})
\ar[r] \ar[d] & 0 \\
0 \ar[r] & \overline{\xi}_m^{\overline{E}} \ar[r] \ar[d] & \cdots
\ar[r] & \overline{\xi}_1^{\overline{E}} \ar[r] \ar[d] &
\overline{\xi}_0^{\overline{E}} \ar[r] \ar[d] & (f_*E,f_*h^{E})
\ar[r] \ar[d] & 0 \\
0 \ar[r] & \overline{\xi}_m^{\overline{E}''}  \ar[r] \ar[d] &
\cdots \ar[r] & \overline{\xi}_1^{\overline{E}''} \ar[r] \ar[d] &
\overline{\xi}_0^{\overline{E}''} \ar[r] \ar[d] &
(f_*E'',f_*h^{E''})
\ar[r] \ar[d] & 0 \\
 & 0  & & 0  & 0  & 0  & .}
\end{displaymath} According to the generating relations of
$\widehat{K_0}(Y,\mu_n)$, our definition of the push-forward
morphism $f_*$ and the double complex formula of equivariant
secondary Bott-Chern classes, we have an equation
\begin{displaymath}
f_*\overline{E}'-f_*\overline{E}+f_*\overline{E}''=-T_g(\omega^X,h^{E'})+T_g(\omega^X,h^E)-T_g(\omega^X,h^{E''})+\widetilde{{\rm
ch}}_g(f_*\overline{\varepsilon}) \end{displaymath} in
$\widehat{K_0}(Y,\mu_n)$. On the other hand, we apply Bismut-Ma's
immersion formula to the case where the fibration is with respect
to $f_\C: f_\C^{-1}(Y_g)\rightarrow Y_g$ and the closed immersion
is the identity map, then the equality
\begin{displaymath}
T_g(\omega^X,h^{E'})-T_g(\omega^X,h^E)+T_g(\omega^X,h^{E''})-\widetilde{{\rm
ch}}_g(f_*\overline{\varepsilon})=-\int_{X_g/{Y_g}}{\rm
Td}_g(Tf,h^{Tf})\widetilde{{\rm ch}}_g(\overline{\varepsilon})
\end{displaymath} holds in $\widetilde{A}(Y_{\mu_n})$. Combing
these two computations above, we finally get
\begin{displaymath}
f_*\overline{E}'-f_*\overline{E}+f_*\overline{E}''=\int_{X_g/{Y_g}}{\rm
Td}_g(Tf,h^{Tf})\widetilde{{\rm ch}}_g(\overline{\varepsilon}).
\end{displaymath} This final expression means that the push-forward
morphism $f_*$ is compatible with the first generating relation of
arithmetic $K_0$-theory. For the second one, it is rather clear
from definition. It is easily seen from the generating relation of
$\widehat{K_0}(Y,\mu_n)$ that $f_*0$ is equal to $0$. Moreover,
since any two resolutions of $f_*E$ are dominated by a third one,
our arguments above also show that $f_*\overline{E}$ is
independent of the choice of resolution. So we are done.
\end{proof}

\begin{lem}\label{pfp}(Projection formula)
For any elements $y\in \widehat{K_0}(Y,\mu_n)$ and $x\in
\widehat{K_0}(X,\mu_n)$, the equality $f_*(f^*y\cdot x)=y\cdot
f_*x$ holds in $\widehat{K_0}(Y,\mu_n)$.
\end{lem}
\begin{proof}
Assume that $y=\overline{E}$ is an equivariant hermitian vector
bundle and $x=\overline{F}$ is an $f$-acyclic equivariant
hermitian vector bundle, then $f^*y\cdot x=f^*\overline{E}\otimes
\overline{F}$. By projection formula for direct images and the
definition of $L^2$-metric, we know that $(f_*(f^*E\otimes
F),f_*h^{f^*E\otimes F})$ is isometric to $\overline{E}\otimes
(f_*F,f_*h^F)$. Moreover, concerning the analytic torsion form, we
have $T_g(\omega^X,h^{f^*E\otimes F})={\rm
ch}_g(\overline{E})T_g(\omega^X,h^F)$. So the projection formula
$f_*(f^*y\cdot x)=y\cdot f_*x$ holds in this case.

Assume that $y=\overline{E}$ is an equivariant hermitian vector
bundle and $x=\alpha$ is represented by some smooth form. We write
$f_g^*$ and ${f_g}_*$ for the pull-back and push-forward of smooth
forms respectively, then \begin{eqnarray*} f_*(f^*y\cdot
x)&=&f_*(f_g^*{\rm
ch}_g(\overline{E})\alpha)\\
&=&{f_g}_*(f_g^*{\rm
ch}_g(\overline{E})\alpha{\rm Td}_g(\overline{Tf}))={\rm
ch}_g(\overline{E}){f_g}_*(\alpha{\rm
Td}_g(\overline{Tf}))\\
&=&{\rm
ch}_g(\overline{E})\int_{X_g/{Y_g}}\alpha{\rm
Td}_g(\overline{Tf})=y\cdot f_*x.
\end{eqnarray*}
Here we have used the projection formula of smooth forms
$p_*(p^*\alpha_1\wedge\alpha_2)=\alpha_1\wedge p_*\alpha_2$ (cf.
\cite{GHV}, Prop. IX p. 303).

Assume that $y=\beta$ is represented by some smooth form and
$x=\overline{F}$ is an $f$-acyclic hermitian vector bundle, then
\begin{eqnarray*}
f_*(f^*y\cdot x)&=&f_*(f_g^*\beta{\rm
ch}_g(\overline{F}))={f_g}_*(f_g^*\beta{\rm
ch}_g(\overline{F}){\rm
Td}_g(\overline{Tf}))\\
&=&\beta{f_g}_*({\rm
ch}_g(\overline{E}){\rm
Td}_g(\overline{Tf}))=\beta\int_{X_g/{Y_g}}{\rm
ch}_g(\overline{E}){\rm Td}_g(\overline{Tf})\\
&=&\beta{\rm
ch}_g(\overline{f_*{F_\C}})-\beta{\rm dd}^cT_g(\omega^X,h^F)
\end{eqnarray*}
which is exactly $y\cdot f_*x$.

Finally, assume that $y=\beta$ and $x=\alpha$ are both represented
by smooth forms, then \begin{displaymath} f_*(f^*y\cdot
x)=f_*(f_g^*\beta({\rm dd}^c\alpha))={f_g}_*(f_g^*\beta({\rm
dd}^c\alpha){\rm Td}_g(\overline{Tf}))=\beta{\rm
dd}^c{f_g}_*(\alpha{\rm Td}_g(\overline{Tf})) \end{displaymath}
which is also equal to $y\cdot f_*x$.

Since $f_*$ and $f^*$ are both group homomorphisms, we may
conclude the projection formula by linear extension.
\end{proof}

\begin{rem}\label{pfprem}
Lemma~\ref{pfp} implies that $f_*$ is a homomorphism of
$R(\mu_n)$-modules, and hence it induces a push-forward morphism
after taking localization.
\end{rem}

We next translate Bismut-Ma's immersion formula to a
$\widehat{K_0}$-theoretic version. Let notations and assumptions
be as before. Assume that the morphism $f: X\rightarrow Y$ has a
factorization $f=l\circ i$ such that $i: X\rightarrow P$ is a
closed immersion of equivariant arithmetic varieties and $l:
P\rightarrow Y$ is still equivariant and smooth on the complex
numbers. We additionally suppose that the $\mu_n$-action on $Y$ is
trivial. Let \begin{displaymath} \overline{\Xi}:\quad 0\rightarrow
\overline{\xi}_m\rightarrow
\overline{\xi}_{m-1}\rightarrow\cdots\overline{\xi}_0\rightarrow
i_*\overline{\eta}\rightarrow 0 \end{displaymath} be a resolution
by $l$-acyclic equivariant hermitian vector bundles on $P$ of an
$f$-acyclic equivariant hermitian vector bundle $\overline{\eta}$
on $X$. We suppose that the K\"{a}hler metric on $X$ is induced by
a $\mu_n(\C)$-invariant K\"{a}hler metric on $P$ and the normal
bundle $N$ of $i(X)$ in $P$ carries the quotient metric. As usual,
suppose that the metrics on $\xi.$ satisfy Bismut assumption (A).

\begin{thm}\label{503}
Let notations and assumptions be as above. Then the equality
\begin{eqnarray*}
f_*(\overline{\eta})-\sum_{j=0}^m(-1)^jl_*(\overline{\xi}_j)&=&\int_{X_g/{Y}}{\rm
ch}_g(\eta)R_g(N){\rm
Td}_g(Tf)+\int_{P_g/{Y}}T_g(\overline{\xi}.){\rm
Td}_g(Tl,h^{Tl})\\
&&+\int_{X_g/{Y}}{\rm ch}_g(\overline{\eta})\widetilde{{\rm
Td}}_g(\overline{Tf},\overline{Tl}\mid_X){\rm
Td}_g^{-1}(\overline{N})
\end{eqnarray*}
holds in $\widehat{K_0}(Y,\mu_n)$.
\end{thm}
\begin{proof}
The verification follows rather directly from the generating
relations of arithmetic $K_0$-theory. In fact
\begin{eqnarray*}
f_*(\overline{\eta})-\sum_{j=0}^m(-1)^jl_*(\overline{\xi}_j)&=&
(f_*\eta,f_*h^\eta)-T_g(\omega^X,h^{\eta})-\sum_{j=0}^m(-1)^j((l_*\xi_j,l_*h^{\xi_j})-T_g(\omega^P,h^{\xi_j}))\\
&=&\widetilde{{\rm
ch}}_g(l_*\overline{\Xi}_\C)-T_g(\omega^X,h^{\eta})+\sum_{j=0}^m(-1)^jT_g(\omega^P,h^{\xi_j}).
\end{eqnarray*}
And the right-hand side of the last equality is exactly the
left-hand side of Bismut-Ma's immersion formula. One just need to
note that to simplify the right hand-side of Bismut-Ma's immersion
formula, we have used an Atiyah-Segal-Singer type formula for
immersion \begin{displaymath} {i_g}_*({\rm Td}_g^{-1}(N){\rm
ch}_g(x))={\rm ch}_g(i_*(x)). \end{displaymath} This formula is
the content of \cite{KR1}, Theorem 6.16.
\end{proof}

\begin{rem}\label{additional2}
The $\widehat{K_0}$-theoretic version of Bismut-Ma's immersion
formula even holds without acyclicity conditions on the bundles
$\eta$ and $\xi_j$. One can carry out the same principle of the
proof of \cite{Roe}, Theorem 6.7 to show this fact because every
equivariant arithmetic variety is supposed to admit a
$\mu_n$-projective action.
\end{rem}

We now turn to the description of the relative fixed point formula
of Lefschetz type and its proof. Let $f: X\rightarrow Y$ be an
equivariant morphism of $\mu_n$-equivariant arithmetic varieties,
which is flat and smooth over the complex numbers. We additionally
suppose that the fibre product $f^{-1}(Y_{\mu_n})$ is regular. We
shall naturally endow $X_{\mu_n}(\C)$ and $f^{-1}(Y_{\mu_n})(\C)$
with the K\"{a}hler metric induced by the K\"{a}hler metric of
$X(\C)$.

We have the following Cartesian square
\begin{displaymath}
\xymatrix{f^{-1}(Y_{\mu_n}) \ar[r] \ar[d] & X \ar[d] \\
Y_{\mu_n} \ar[r] & Y} \end{displaymath} whose rows are both closed
immersions of $\mu_n$-equivariant arithmetic varieties and whose
columns are both flat. Then the normal bundle
$N_{X/{f^{-1}(Y_{\mu_n})}}$ is isomorphic to the pull-back of
normal bundle $f^*N_{Y/{Y_{\mu_n}}}$. Therefore by restricting to
$X_{\mu_n}$ we get a short exact sequence
\begin{displaymath}
\overline{\mathcal{N}}:\quad 0\rightarrow
\overline{N}_{f^{-1}(Y_{\mu_n})/X_{\mu_n}}\rightarrow
\overline{N}_{X/{X_{\mu_n}}}\rightarrow
f^*\overline{N}_{Y/{Y_{\mu_n}}}\rightarrow 0 \end{displaymath} of
equivariant hermitian vector bundles. Here all metrics on these
normal bundles are the quotient metrics. Define
\begin{eqnarray*}
M(f):=\big(\lambda_{-1}^{-1}(\overline{N}_{X/{X_{\mu_n}}}^\vee)\lambda_{-1}(f^*\overline{N}_{Y/{Y_{\mu_n}}}^\vee)+\widetilde{{\rm
Td}}_g(\overline{\mathcal{N}}){\rm
Td}_g^{-1}(f^*\overline{N}_{Y/{Y_{\mu_n}}})\big)\\
\cdot\big(1-R_g(N_{X/{X_{\mu_n}}})+R_g(f^*N_{Y/{Y_{\mu_n}}})\big).
\end{eqnarray*}
Note that $f_{\mu_n}$ is still smooth over the complex numbers
because ${f_{\mu_n}}_\C=f_g$ is still a submersion. Then we can
define the push-forward morphism ${f_{\mu_n}}_*$.

\begin{thm}\label{504}(Relative fixed point formula)
Let notations and assumptions be as above. Then the following
diagram \begin{displaymath} \xymatrix{\widehat{K_0}(X,\mu_n)
\ar[rr]^-{M(f)\cdot\tau} \ar[d]^{f_*} &&
\widehat{K_0}(X_{\mu_n},\mu_n)_{\rho}
\ar[d]^{{f_{\mu_n}}_*} \\
\widehat{K_0}(Y,\mu_n) \ar[rr]^-{\tau} &&
\widehat{K_0}(Y_{\mu_n},\mu_n)_{\rho}} \end{displaymath} commutes,
where $\tau$ stands for the restriction map.
\end{thm}
\begin{proof}
Since $f$ is flat, the Cartesian square introduced before this
theorem induces a commutative diagram in arithmetic $K_0$-theory
\begin{displaymath}
\xymatrix{\widehat{K_0}(X,\mu_n) \ar[r]^-{\tau} \ar[d]^{f_*} &
\widehat{K_0}(f^{-1}(Y_{\mu_n}),\mu_n)
\ar[d]^{f_*} \\
\widehat{K_0}(Y,\mu_n) \ar[r]^-{\tau} &
\widehat{K_0}(Y_{\mu_n},\mu_n).} \end{displaymath} From the exact
sequence \begin{displaymath} \overline{\mathcal{N}}:\quad
0\rightarrow \overline{N}_{f^{-1}(Y_{\mu_n})/X_{\mu_n}}\rightarrow
\overline{N}_{X/{X_{\mu_n}}}\rightarrow
f^*\overline{N}_{Y/{Y_{\mu_n}}}\rightarrow 0 \end{displaymath} we
know that
\begin{displaymath}
\lambda_{-1}^{-1}(\overline{N}_{f^{-1}(Y_{\mu_n})/X_{\mu_n}}^\vee)\lambda_{-1}^{-1}(f^*\overline{N}_{Y/{Y_{\mu_n}}}^\vee)-
\lambda_{-1}^{-1}(\overline{N}_{X/{X_{\mu_n}}}^\vee)=\widetilde{{\rm
Td}}_g(\overline{\mathcal{N}}) \end{displaymath} according to
\cite{KR1}, Lemma 7.1 and that
\begin{displaymath}
R_g(N_{f^{-1}(Y_{\mu_n})/X_{\mu_n}})+R_g(f^*N_{Y/{Y_{\mu_n}}})=R_g(N_{X/{X_{\mu_n}}})
\end{displaymath} since the equivariant $R$-genus is additive.
Therefore, we may reduce our proof to the case where the
$\mu_n$-action on the base variety $Y$ is trivial. In this case,
denote by $i$ the canonical closed immersion $X_{\mu_n}\rightarrow
X$. Then for any element $x\in \widehat{K_0}(X,\mu_n)_{\rho}$, we
have
\begin{displaymath}
f_*(x)=f_*i_*i_*^{-1}(x)=f_*i_*(\lambda_{-1}^{-1}(\overline{N}_{X/{X_{\mu_n}}}^\vee)\cdot
\tau(x)). \end{displaymath} Consider the factorization
$f_{\mu_n}=f\circ i$, we have to compute the difference
${f_{\mu_n}}_*-f_*i_*$ in arithmetic $K_0$-theory. Indeed, this
difference can be measured by Bismut-Ma's immersion formula. By
applying Theorem~\ref{503} and Remark~\ref{additional2} to the
closed immersion $i$, for any equivariant hermitian vector bundle
$\overline{\eta}$ on $X_{\mu_n}$ we have
\begin{displaymath}
{f_{\mu_n}}_*(\overline{\eta})-f_*i_*(\overline{\eta})=\int_{X_g/Y}{\rm
ch}_g(\eta)R_g(N_{X/{X_{\mu_n}}}){\rm
Td}_g(Tf_{\mu_n})={f_{\mu_n}}_*({\rm
ch}_g(\eta)R_g(N_{X/{X_{\mu_n}}})). \end{displaymath} The first
equality holds because the exact sequence
\begin{displaymath}
0\rightarrow \overline{Tf_g}\rightarrow
\overline{Tf}\mid_{X_g}\rightarrow
\overline{N}_{X/{X_g}}\rightarrow 0 \end{displaymath} is
orthogonally split on $X_g$ so that $\widetilde{{\rm
Td}}_g(\overline{Tf_g},\overline{Tf}\mid_{X_g})=0$. The second
equality follows from \cite{KR1}, Lemma 7.3 and the fact that
${\rm ch}_g(\eta)R_g(N_{X/{X_{\mu_n}}})$ is ${\rm dd}^c$-closed.
On the other hand, let $\alpha$ be an element in
$\widetilde{A}(X_{\mu_n})$, we have
\begin{eqnarray*}
{f_{\mu_n}}_*(\alpha)-f_*i_*(\alpha)&=&\int_{X_g/Y}\alpha{\rm
Td}_g(Tf_{\mu_n},h^{Tf_{\mu_n}})-\int_{X_g/Y}\alpha{\rm
Td}_g^{-1}(\overline{N}_{X/{X_{\mu_n}}}){\rm
Td}_g(Tf,h^{Tf})\\
&=&\int_{X_g/Y}\alpha{\rm
Td}_g^{-1}(\overline{N}_{X/{X_{\mu_n}}}){\rm dd}^c\widetilde{{\rm
Td}}_g(\overline{Tf_g},\overline{Tf}\mid_{X_g})=0.
\end{eqnarray*}
Combing these two computations, by the ring structure of
$\widehat{K_0}(\cdot)$, we know that the equality
\begin{displaymath}
{f_{\mu_n}}_*(y)-f_*i_*(y)={f_{\mu_n}}_*(y\cdot
R_g(N_{X/{X_{\mu_n}}})) \end{displaymath} holds for any element
$y\in \widehat{K_0}(X_{\mu_n},\mu_n)_\rho$ since both two sides
are additive. Now continue our computation for $f_*(x)$, we obtain
that
\begin{eqnarray*}
f_*(x)&=&{f_{\mu_n}}_*(\lambda_{-1}^{-1}(\overline{N}_{X/{X_{\mu_n}}}^\vee)\cdot
\tau(x))-{f_{\mu_n}}_*(\lambda_{-1}^{-1}(\overline{N}_{X/{X_{\mu_n}}}^\vee)\cdot
\tau(x)\cdot R_g(N_{X/{X_{\mu_n}}}))\\
&=&{f_{\mu_n}}_*(\lambda_{-1}^{-1}(\overline{N}_{X/{X_{\mu_n}}}^\vee)\cdot(1-R_g(N_{X/{X_{\mu_n}}}))\cdot\tau(x)).
\end{eqnarray*}
The last thing should be indicated is that the following square
\begin{displaymath}
\xymatrix{\widehat{K_0}(X,\mu_n) \ar[r]^-{\iota} \ar[d]^{f_*}
& \widehat{K_0}(X,\mu_n)_{\rho} \ar[d]^{f_*} \\
\widehat{K_0}(Y,\mu_n) \ar[r]^-{\iota} &
\widehat{K_0}(Y,\mu_n)_{\rho}}
\end{displaymath} is naturally commutative. So the relative fixed
point formula holds in the case where the $\mu_n$-action on the
base variety $Y$ is trivial. By the observation given at the
beginning of this proof, it is enough to conclude the statement in
our theorem.
\end{proof}

\begin{rem}\label{505}
(i). The condition $f^{-1}(Y_{\mu_n})$ is regular can be satisfied
if the $\mu_n$-action on the base variety is trivial or the
morphism $f$ is not only smooth over the complex numbers but also
smooth everywhere. This is already enough for the applications in
practice. For instance, our main result implies various formulae
stated as conjectural in \cite{MR2}, in particular Proposition 2.3
in that article.

(ii). The condition of flatness on the morphism $f$ is only used
in the reduction of general case to the case where the
$\mu_n$-action on the base variety is trivial. If the
$\mu_n$-action on $Y$ is trivial, one can certainly remove the
condition of flatness.
\end{rem}

\hspace{5cm} \hrulefill\hspace{5.5cm}

D\'{e}partement de Math\'{e}matiques, B\^{a}timent 425,
Universit\'{e} Paris-Sud, 91405 Orsay cedex, France

E-mail: shun.tang@math.u-psud.fr

\end{document}